\documentclass[10pt,letterpaper]{amsart} 

\usepackage{color}
\usepackage[margin=0.95in]{geometry}
\usepackage{subfig}
\usepackage{multicol}
\usepackage{listings}
\usepackage{amsmath, amsthm, amssymb, wasysym, verbatim, bbm, color, graphics}
\usepackage{url}
\usepackage{mathtools}
\usepackage{hyperref}
\hypersetup{colorlinks,allcolors=blue}
\usepackage{romannum}

\usepackage{xfrac}
\usepackage{float}

\usepackage[title]{appendix}
\usepackage{todonotes}
\newcommand{\R}{\mathbb{R}}

\newcommand{\N}{\mathbb{N}}

\DeclarePairedDelimiter{\ceil}{\lceil}{\rceil}
\DeclarePairedDelimiter{\floor}{\lfloor}{\rfloor}

\newcommand{\hlfstp}{n+\frac{1}{2}}

\DeclareMathOperator{\tr}{\text{tr}}

\newcommand\innprod[2]{\langle #1,#2 \rangle}
\newcommand\norm[1]{\left\lVert#1\right\rVert}
\newcommand{\Grad}{\nabla}
\newcommand{\Div}{\operatorname{div}}
\newcommand{\dom}{\Omega}

\newtheorem{lemma}{Lemma}[section]
\newtheorem{proposition}[lemma]{Proposition}
\newtheorem{corollary}[lemma]{Corollary}
\newtheorem{theorem}[lemma]{Theorem}

\newtheorem{notation}[lemma]{Notation}

\newtheorem*{maintheorem*}{Main Theorem}
\theoremstyle{definition}{\newtheorem{definition}[lemma]{Definition}}

\newtheorem{thm}[lemma]{Theorem}

\newtheorem{lem}[lemma]{Lemma}

\allowdisplaybreaks

\numberwithin{equation}{section}

\title[Numerics for the Q-tensor flow]{Convergence analysis of a fully discrete energy-stable numerical scheme for the Q-tensor flow of liquid crystals}
\date{\today}

\author[V. M. Gudibanda]{Varun M. Gudibanda} \address[Varun M. Gudibanda]{\newline Department of Mathematics
	\newline University of Wisconsin -- Madison\newline
	A0048 - 480 Lincoln Dr, 
	Madison, WI 53706-1325, USA.}
\email[]{gudibanda@wisc.edu}

\author[F. Weber]{Franziska Weber}
\address[Franziska Weber]{\newline Department of Mathematical Sciences \newline Carnegie Mellon University \newline 5000 Forbes Avenue, Pittsburgh, PA 15213, USA.}
\email[]{franzisw@andrew.cmu.edu}
\author[Y. Yue]{Yukun Yue}
\address[Yukun Yue]{\newline Department of Mathematical Sciences \newline Carnegie Mellon University \newline 5000 Forbes Avenue, Pittsburgh, PA 15213, USA.}
\email[]{yukuny@andrew.cmu.edu}

\begin{document}
 \pagenumbering{arabic}
 \begin{abstract}
 	We present a fully discrete convergent finite difference scheme for the Q-tensor flow of liquid crystals based on the energy-stable semi-discrete scheme by Zhao, Yang, Gong, and Wang (Comput. Methods Appl. Mech. Engrg. 2017). We prove stability properties of the scheme and show convergence to weak solutions of the Q-tensor flow equations. We demonstrate the performance of the scheme in numerical simulations.
 \end{abstract}
\maketitle

\section{Introduction}

Liquid crystals constitute a state of matter that is intermediate between solids and liquids. On one hand, they have properties that are typical for fluids, in particular they have the ability to flow, on the other hand, they exhibit properties of solids, as an example, their molecules are oriented in a crystal-like manner. A common characteristic of materials exhibiting a liquid crystal phase is that they consist of elongated molecules of identical size. They may be pictured as `rods' or `ribbons' and are subject to molecular interactions that make them align alike \cite{sv2012}. 
 
Liquid crystals play an important role in nature: As an example, phospholipids which constitute the  main component of cell membranes, are a form of liquid crystal. They also appear in many daily applications, such as soaps, shampoos and detergents. Further applications include displays of electronic devices (LCD), where one makes use of the optical properties of liquid crystals in the presence or absence of an electric field; thermometers, optical switches~\cite{Goodby1989,Wissbrun1984}, and biotechnological applications.
One generally distinguishes three types of liquid crystals, nematics, cholesterics and smectics. We focus here on the numerical discretization of a liquid crystal model for nematic liquid crystals, the so-called Q-tensor model.
\subsection{Q-tensor model}
In the Q-tensor model by Landau and de Gennes~\cite{deGennes1995}, the main orientation of the liquid crystal molecules is represented by the Q-tensor, a symmetric, trace-free matrix that is assumed to minimize the Landau-de Gennes free energy

\begin{equation*}
E_{LG}(Q)=\int_{\dom}\mathcal{F}_B(Q) +\mathcal{F}_E(Q),
\end{equation*}
in equilibrium situations. Here $\Omega\in \R^d$, $d=2,3$, is the spatial domain occupied by the liquid crystal molecules, $\mathcal{F}_B$ is a bulk potential and $\mathcal{F}_E$ is the elastic energy given by
\begin{equation*}
	\begin{split}
	\mathcal{F}_B(Q) = \frac{a}{2}\tr(Q^2)-\frac{b}{3}\tr(Q^3)+\frac{c}{4}(\tr(Q^2))^2,
	\quad 
	\mathcal{F}_E(Q)= \frac{L_1}{2}|\Grad Q|^2 + \frac{L_2}{2}|\Div Q|^2 + \frac{L_3}{2}\sum_{i,j,k=1}^d \partial_i Q_{jk}\partial_k Q_{ji},	
	\end{split}\end{equation*}
where $a,b,c,L_1,L_2,L_3$ are constants with $c,L_1,L_2,L_3>0$.

Non-equilibrium situations can be described by the gradient flow~\cite{Ball2017,Mottram2014}, 
\begin{equation}\label{eq:qtensorflow}
\begin{split}
\frac{\partial Q_{ij}}{\partial t} = M &\Bigg(L_1 \Delta Q_{ij} + \frac{L_2+L_3}{2} \left(\sum_{k=1}^d (\partial_{ik}Q_{jk} + \partial_{jk}Q_{ik}) - \frac{2}{d} \sum_{k,\ell=1}^d \partial_{k\ell}Q_{k\ell} \delta_{ij}\right) \\
&- \left(a Q_{ij} - b\left((Q^2)_{ij} -\frac{1}{d} \tr(Q^2)\delta_{ij}\right)+c \tr(Q^2)Q_{ij}\right) \Bigg),
\end{split}
\end{equation}
where $M>0$ is a constant, and one approach to obtaining equilibrium states is to follow this gradient flow. 
Adding the dynamics of the mean flow of the liquid crystal fluid to this, one obtains the Beris-Edwards system~\cite{Beris1994}. 

Analysis of the Q-tensor flow has been done, e.g.,~\cite{Iyer2015,Contreras2019,WangWang2017}, and numerical methods for the Q-tensor flow have been constructed in~\cite{ZhaoYang2017,shen2019,Mori1999,Tovkach2017}. To the best of our knowledge, none of these methods has been shown to be convergent to a weak solution of~\eqref{eq:qtensorflow}. An exception is the work by Cai, Shen and Xu~\cite{Cai2017}, where under the assumption of smallness of the initial data, convergence of a time discretization in 2D is proved. Our goal is to show convergence to weak solutions of a fully discrete method for~\eqref{eq:qtensorflow} in 2 and 3D under only the natural assumption that the initial energy is bounded. Our numerical method is based on the invariant energy quadratization idea (IEQ) by Zhao et al.~\cite{ZhaoYang2017} which we combine with a finite difference discretization in space. 
 
This method takes as a basis the reformulation of the Q-tensor flow using the auxiliary variable $r$:
\begin{equation}\label{eq:defr}
r(Q)=\sqrt{2\left(\frac{a}{2}\tr(Q^2)-\frac{b}{3}\tr(Q^3)+\frac{c}{4}\tr^2(Q^2)+A_0\right)},
\end{equation}
where $A_0>0$ is a constant ensuring that $r$ is positive. Defining
\begin{equation}\label{eq:defS}
S(Q)=a Q-b\left[Q^2-\frac{1}{d}\text{tr}(Q^2)I\right]+c\text{tr}(Q^2)Q,
\end{equation}
it follows that
\begin{equation}\label{eq:defP}
\frac{\delta r(Q)}{\delta Q} = \frac{S(Q)}{r(Q)}:= P(Q),
\end{equation}
for symmetric, trace free tensors $Q$.
Then one can formally write the gradient flow~\eqref{eq:qtensorflow} as a system for $(Q,r)$:
\begin{subequations}\label{eq:reformulation}
	\begin{align}
	Q_t &=M\left(L_1\Delta Q +\frac{L_2+L_3}{2}\alpha(Q)-r P(Q)\right):=MH,\label{eq:QZ}\\
	r_t & = P(Q):Q_t,\label{eq:rZ}
	\end{align}
\end{subequations}
where 
\begin{equation*}
\alpha(Q)_{ij} = \sum_{k=1}^d (\partial_{ik}Q_{jk} + \partial_{jk}Q_{ik}) - \frac{2}{d} \sum_{k,\ell=1}^d \partial_{k\ell}Q_{k\ell} \delta_{ij}.
\end{equation*}
It is easy to see that this reformulation comes with a formal energy law: Multiplying the first equation~\eqref{eq:QZ} with $-H$ and~\eqref{eq:rZ} with $r$, adding and integrating, and integrating by parts, we obtain
\begin{equation}\label{eq:energyQr}
\frac{d}{dt}\,\frac12\int_\Omega\left(L_1|\Grad Q|^2+(L_2+L_3)|\Div Q|^2 +r^2\right)dx = -M\int_\Omega |H|^2 dx.
\end{equation}
In~\cite{ZhaoYang2017}, a time discretization of the system~\eqref{eq:reformulation} is proposed that retains a discrete version of the energy law~\eqref{eq:energyQr}. Based on this prior work, we propose a fully discrete finite difference method for~\eqref{eq:reformulation} and prove its convergence to weak solutions of~\eqref{eq:reformulation} as defined in the following Definition~\ref{def:weakreformulation}.

We then proceed to showing that weak solutions of~\eqref{eq:reformulation} are in fact weak solutions of~\eqref{eq:qtensorflow} and so achieve convergence to the original system~\eqref{eq:qtensorflow}. To the best of our knowledge, this is the first convergence proof for a fully discrete numerical scheme discretizing~\eqref{eq:qtensorflow}.
 The proof is based on the derivation of discrete energy stability of the fully discrete scheme, then using this to derive the existence of a precompact sequence that allows us to pass to the limit in the approximations. We proceed to showing Lipschitz continuity of the function $P$ and use a Lax-Wendroff type argument to show that the limit of the approximating sequence is a weak solution of~\eqref{eq:reformulation}. The last step is to show that weak solutions of~\eqref{eq:reformulation} are in fact weak solutions of~\eqref{eq:qtensorflow}. We achieve this through showing that a weak form of the chain rule holds in this case. We conclude with numerical experiments in 2D. Our scheme and analysis is for the 3D case but adaptions to 2D can be made easily.

\section{Preliminaries}
\begin{notation}
We introduce the following general notation for matrix-valued functions $A,B: \R^d \to \R^{d \times d}$:
	\begin{multicols}{2}
\begin{itemize}

	\item $A:B = \sum_{i,j=1}^d A_{ij}B_{ij}$,
	\item $\innprod{A}{B} = \int_{\Omega} A:B\, dx$,
	\item $|A|:=|A|_F = \sqrt{A:A}$,
	\item $\norm{A}^2_{L^2} = \int_{\Omega} |A|_F^2\, dx$, 
	\item $\partial_i A = (\partial_i A_{jk})_{jk}$, $\partial_i=\partial_{x_i}$,
	
	\item $\nabla A=(\partial_1 A,\dots,\partial_d A)$,
	\item $|\nabla A|^2 =\sum_{i=1}^d |\partial_i A|_F^2$,
	\item $\norm{\nabla A}_{L^2}^2 = \int_{\Omega}|\nabla A|^2\, dx$.
\end{itemize}
	\end{multicols}
\end{notation}

	We assume $\Omega\subset\R^d$ is a bounded, connected domain with Lipschitz boundary and $Q_0:\Omega\to\R^{d\times d}\in (H^1(\dom))^{d\times d}$ takes values in the symmetric trace-free $d\times d$ matrices and satisfies
	$\left.Q_0\right|_{\partial\Omega} = 0$.
	 Fix $T>0$ an arbitrary time horizon. We then define weak solutions of~\eqref{eq:qtensorflow} as follows:
	\begin{definition}
		\label{def:weakQ}
		By a weak solution of~\eqref{eq:qtensorflow}, we mean a function $Q:[0,T]\times\Omega\to\R^{d\times d}$ that is trace-free and symmetric for every $(t,x)$ and satisfies
		\begin{equation*}
		Q\in L^\infty(0,T;H^1(\dom)),\quad Q_t\in L^2([0,T]\times\Omega),
		\end{equation*}
		and
		\begin{equation}
		\label{eq:weakform}
		\begin{split}
		&\int_0^T\int_\Omega Q :\partial_t\varphi dx dt-\int_\Omega Q(T,x):\varphi(T,x) dx +\int_{\Omega}Q_0(x):\varphi(0,x) dx\\
		& =M \int_0^T\int_\Omega\left(L_1 \sum_{i,j=1}^d\Grad Q_{ij}\cdot\Grad\varphi_{ij}  + \frac{L_2+L_3}{2} \sum_{i,j,k=1}^d \left(\partial_{k}Q_{jk}\partial_i \varphi_{ij} + \partial_{k}Q_{ik}\partial_j\varphi_{ij} -\frac{2}{d} \, \partial_{i}Q_{ki} \partial_k\varphi_{jj}\right)\right)dxdt \\
		&\quad+M\int_0^T\int_\Omega \left(a Q - b\left((Q^2) -\frac{1}{d} \tr(Q^2)I\right)+c \tr(Q^2)Q\right):\varphi \, dxdt,
		\end{split}
		\end{equation}
		for all smooth $\varphi = (\varphi_{ij})_{i,j=1}^d:[0,T]\times \Omega\to \R^{d\times d}$ that are compactly supported within $\dom$ for almost every $t\in [0,T]$. Furthermore, $Q$ satisfies the energy inequality
		\begin{multline}\label{eq:weakform_energy_inequality}
		 \frac12\int_\Omega L_1|\nabla Q(t,x)|^2  +( L_2+L_3)|\Div Q(t,x)|^2+2\mathcal{F}_B\left(Q(t,x)\right)\,dx\\
		 \leq \frac12\int_\Omega L_1|\nabla Q_0|^2+(L_2+L_3)|\Div Q_0|^2+2\mathcal{F}_B\left(Q_0\right)\,dx-M\int_0^t\int_\Omega |H(s,x)|^2\,dx\,ds,
	\end{multline}
		for every $t\in[0,T]$.
	\end{definition}
Similarly, we define weak solutions of the reformulation~\eqref{eq:reformulation}:
\begin{definition}
	\label{def:weakreformulation}
		By a weak solution of~\eqref{eq:reformulation}, we mean a pair of functions $Q:[0,T]\times\Omega\to\R^{d\times d}$ and $r:[0,T]\times\dom\to\R$, with $Q(t,x)$ trace-free and symmetric for every $(t,x)$, and satisfying
	\begin{equation*}
	Q\in L^\infty(0,T;H^1(\dom)),\quad Q_t\in L^2([0,T]\times\Omega), \quad r\in L^\infty(0,T;L^2(\dom))
	\end{equation*}
	and
	\begin{equation}
	\label{eq:weakformQr}
	\begin{split}
	&\int_0^T\int_\Omega Q :\partial_t\varphi dx dt-\int_\Omega Q(T,x):\varphi(T,x) dx +\int_{\Omega}Q_0(x):\varphi(0,x) dx\\
	& = M \int_0^T\int_\Omega\left(L_1 \sum_{i,j=1}^d\Grad Q_{ij}\cdot\Grad\varphi_{ij}  + \frac{L_2+L_3}{2} \sum_{i,j,k=1}^d \left(\partial_{k}Q_{jk}\partial_i \varphi_{ij} + \partial_{k}Q_{ik}\partial_j\varphi_{ij} -\frac{2}{d} \, \partial_{i}Q_{ki} \partial_k\varphi_{jj}\right)\right)dxdt \\
	&\quad+M\int_0^T\int_\Omega r P(Q):\varphi \, dxdt,
	\end{split}
	\end{equation}
	and
	\begin{equation}
	\label{eq:weakr}
	\int_0^T\int_\dom r\, \phi_t dxdt - \int_\dom r(T,x) \phi(T,x) dx +\int_\dom r_0(x)\phi(0,x) dx =- \int_0^T\int_\dom P(Q):Q_t \,\phi \,dx dt,
	\end{equation}
	where $r_0 = r(Q_0)$ and $P(Q)$ is defined in~\eqref{eq:defP}, for all smooth $\varphi = (\varphi_{ij})_{i,j=1}^d:[0,T]\times \Omega\to \R^{d\times d}$ and $\phi:[0,T]\times \dom\to \R$ that are compactly supported within $\dom$ for every $t\in [0,T]$. Furthermore, $(Q,r)$ satisfies for a.e. $t\in [0,T]$ the energy inequality
	\begin{multline}\label{eq:weakr_energy_inequality}
	    \frac12\int_\Omega L_1|\nabla Q(t,x)|^2 +(L_2+L_3) |\Div Q(t,x)|^2+|r(t,x)|^2\,dx\\
\leq 	  \frac12\int_\Omega L_1|\nabla Q_0|^2 +(L_2+L_3) |\Div Q_0|^2+|r_0|^2\,dx-M\int_0^t\int_\Omega |H(s,x)|^2\,dx\,ds 
	\end{multline}

\end{definition}
\section{The numerical scheme}
We start by introducing notation to define our numerical scheme.
We let $\Delta t>0$ be a time step size and $t^n:=n\Delta t$ time levels at which we intend to compute approximations. For the ease of notation, we present the scheme for the case $\Omega =[0,1]^3$ and $h>0$ is a uniform grid size in each spatial dimension. Extensions to square prisms of different side lengths and nonuniform grid sizes are not hard but notationally cumbersome, therefore we restrict our analysis to the cube in $\mathbb{R}^3$ and uniform mesh sizes. The 2D case can easily be derived from the 3D scheme presented here. 
 We let $x_{ijk}=(x_i,y_j,z_k)=(i h,j h,k h)$ be grid points, $i,j,k=0,\dots, N+1$, with $N+1=1/h\in \N$.
For approximations $(f_{ijk})_{ijk=0}^{N+1}$ on this grid, we define the averages
\begin{equation*}
f_{ijk}^{\hlfstp}=\frac{f_{ijk}^{n+1}+f_{ijk}^n}{2},\quad \overline{f_{ijk}}^{\hlfstp} = \frac{3}{2}f_{ijk}^{n} - \frac{1}{2}f_{ijk}^{n-1}
\end{equation*}
 and difference operators:
 \begin{equation}\label{def:operator}
\begin{split}
&\quad D_t^\pm f^n_{ijk} =\pm \frac{f^{n\pm 1}_{ijk}-f^n_{ijk}}{\Delta t}, \\
D^\pm_1 f^n_{ijk} = \pm\frac{f^n_{i\pm1,j,k}-f^n_{ijk}}{h},&\quad  D^\pm_2 f^n_{ijk} = \pm\frac{f^n_{i,j\pm1,k}-f^n_{ijk}}{h},\quad D^\pm_3 f^n_{ijk} = \pm\frac{f^n_{i,j,k\pm1}-f^n_{ijk}}{h},\\
 D^c_1 f^n_{ijk} = \frac{f^n_{i+1,j,k}-f^n_{i-1,j,k}}{2h},&\quad
 D^c_2 f^n_{ijk} = \frac{f^n_{i,j+1,k}-f^n_{i,j-1,k}}{2h},\quad D^c_3 f^n_{ijk} = \frac{f^n_{i,j,k+1}-f^n_{i,j,k-1}}{2h},
\end{split} \end{equation}
for $i,j,k=1,\cdots,N$. 
We will also need the discrete gradient, Laplacian and divergence operators:
\begin{align*}
\Grad_h^\pm f_{ijk}&=(D^\pm_1 f_{ijk},D^\pm_2 f_{ijk},D^\pm_3 f_{ijk})^\top,\\
\Delta_h f_{ijk} = \sum_{\alpha=1}^3 D_\alpha^-D_\alpha^+ f_{ijk},& \quad
 (\Div_h f_{ijk})_\beta = \sum_{\alpha=1}^3 D^c_\alpha (f_{ijk})_{\alpha\beta},
\end{align*}
where $(f_{ijk})_{\alpha\beta}$ is the $(\alpha,\beta)$-entry of the $3\times 3$-matrix $f_{ijk}$.
We approximate the initial data using cell averages:
\begin{equation*}
Q_{ijk}^0 =\frac{1}{h^3}\int_{\mathcal{C}_{ijk}} Q_0(x) dx,\quad r_{ijk}^0 = \frac{1}{h^3}\int_{\mathcal{C}_{ijk}}r(Q_0(x))dx ,\quad  i,j,k=1,\dots ,N,
\end{equation*}
where $\mathcal{C}_{ijk}=[x_i-0.5h,x_i+0.5h)\times[y_j-0.5h,y_j+0.5h)\times [z_k-0.5h,z_k+0.5h)$, for $x_{ijk}=(x_i,y_j,z_k)$,
and use Dirichlet boundary conditions
\begin{equation}
\label{eq:dirchlet}
Q_{0,j,k}=Q_{N+1,j,k}=Q_{i,0,k}=Q_{i,N+1,k}=Q_{i,j,0}=Q_{i,j,N+1}=0,\quad i,j,k=0, 1,\dots,N,N+1.
\end{equation}
For ease of notation throughout, we will also impose boundary conditions on ghost nodes 
\begin{equation}\label{eq:ghostpoints}
    Q_{-1,j,k} = Q_{N+2,j,k} = Q_{i,-1,k} = Q_{i,N+2,k}=Q_{i,j,-1} = Q_{i,j,N+2} =0,\quad i,j,k=0,1,\dots,N,N+1.
\end{equation}
We then propose the following method
\begin{equation} \label{3d_mc_nf_d_scheme}
\begin{cases}
D_t^+Q_{ijk}^n = M\left(L_1 \Delta_h Q_{ijk}^{\hlfstp} - r_{ijk}^{\hlfstp}\overline{P}_{ijk}^{\hlfstp} + \frac{L_2+L_3}{2}\alpha_{ijk}^{\hlfstp}\right) := MH_{ijk}^{\hlfstp} \\
r_{ijk}^{n+1} - r_{ijk}^n = \overline{P}_{ijk}^{\hlfstp}:(Q_{ijk}^{n+1}-Q_{ijk}^n)
\end{cases}
\end{equation}
Here, $Q_{ijk}^n$ is an approximation for $Q$ and $r^n_{ijk}$ is an approximation of $r$ at spatial point $(x_i, y_j, z_k)$ and time step $n$.
We defined $\alpha_{ijk}^{\hlfstp} = \frac{\alpha^{n+1}+\alpha^n}{2}$ where $\alpha^n = \alpha_h(Q_{ijk}^n)$ and $\alpha^{n+1} = \alpha_h(Q_{ijk}^{n+1})$ and $\alpha_h$ is a discretization of $\alpha(Q)$:
\begin{equation} \label{alpha_disc}
\left(\alpha_h(Q_{ijk}^n)\right)_{ws} = \sum_{\beta=1}^3 \left[D_w^c D_\beta^c \left(Q_{ijk}^n\right)_{s\beta} + D_{s}^c D_{\beta}^c \left(Q_{ijk}^n \right)_{w\beta} \right]- \frac{2}{3} \sum_{\beta,\gamma=1}^3 D_{\beta}^c D_{\gamma}^c \left(Q_{ijk}^n\right)_{\beta\gamma} \delta_{ws}
\end{equation}
where the notation $(Q_{ijk}^n)_{ws}$ indicates the element in row $w$ and column $s$ of the matrix $Q_{ijk}^n$. 
\section{Analysis of the numerical scheme}
For the proof of energy stability of this scheme, we will need the following useful lemma which is proved in the appendix (as Lemma~\ref{a:sbp}):
\begin{lemma} \label{sbp}
Let $A_{ijk}$ and $B_{ijk}$ be scalar quantities at grid point $(x_i, y_j, z_k)$ such that $A_{ijk} = 0$ at boundary values, i.e. boundary conditions~\eqref{eq:dirchlet},~\eqref{eq:ghostpoints}.
Then 
\begin{equation*}
\sum_{i,j,k=0}^{N+1} A_{ijk} D_\beta^+ B_{ijk} = -\sum_{i,j,k=0}^{N+1} B_{ijk} D_\beta^- A_{ijk},\quad  
\quad \sum_{i,j,k=0}^{N+1} A_{ijk} D_\beta^- B_{ijk} = -\sum_{i,j,k=0}^{N+1} B_{ijk} D_\beta^+ A_{ijk},
\end{equation*}
and
\begin{equation*}
    \sum_{i,j,k=0}^{N+1} A_{ijk} D_\beta^c B_{ijk} = -\sum_{i,j,k=0}^{N+1} B_{ijk} D_\beta^c A_{ijk}.
\end{equation*}
for $\beta = 1, 2 \text{ or } 3$.
\end{lemma}

\subsection{Energy stability}
We start by defining the following norms and semi-norms for difference approximations. For sequences of approximations $\{f_{ijk}\}$, $\{g_{ijk}\}$, $\{A_{ijk}\}$, and $\{B_{ijk}\}$ of scalar-or vector-valued functions $f,g:\dom\to \R^d$ and matrix-valued functions $A,B:\dom\to \R^{d\times d}$, defined on our grid, we let
\begin{align*}
\innprod{A}{B}_h = h^3 \sum_{i,j,k= 0}^{N+1} A_{ijk}:B_{ijk},&\qquad \innprod{f}{g}_h = h^3 \sum_{i,j,k= 0}^{N+1} f_{ijk}\cdot g_{ijk},\\
\norm{A}_h^2 = h^3 \sum_{i,j,k=0}^{N+1} |A_{ijk}|_F^2,&\qquad \norm{f}_h^2 = h^3 \sum_{i,j,k=0}^{N+1} |f_{ijk}|^2,\\
\norm{\nabla_h A}_{h}^2 &= \sum_{m=1}^d \norm{D^-_m A}_h^2.
\end{align*}
We start by using Lemma~\ref{sbp} to show some simple summation by parts identities that will be useful later in the proofs of energy stability of the scheme.
\begin{lemma}
	\label{lem:sbplaplace}
	Let $\{A_{ijk}\}$ and $\{B_{ijk}\}$ be grid functions satisfying homogeneous Dirichlet boundary conditions ~\eqref{eq:dirchlet}. Then
	\begin{equation*}
	\langle A, \Delta_{h} B\rangle_{h} = -\langle \nabla_{h} A, \nabla_{h} B\rangle_{h} 
	\end{equation*}
\end{lemma}
\begin{proof}
	We write	
	\begin{equation*}
	\innprod{A}{\Delta_h B}_h  =h^3 \sum_{i,j,k=0}^{N+1}\sum_{\alpha=1}^3 A_{ijk}:(D_\alpha^+ D_\alpha^- B_{ijk}) = - h^3\sum_{i,j,k=0}^{N+1}\sum_{\alpha=1}^3 (D_\alpha^- A_{ijk}):( D_\alpha^- B_{ijk})=-\innprod{\Grad_h A}{\Grad_h B}_h,
	\end{equation*} 
	where we used Lemma~\ref{sbp} and the boundary conditions for the second equality.
\end{proof}
For the $\alpha$-term, we have
\begin{lemma}
	\label{lem:sbpalpha}
	Let $\{A_{ijk}\}$ and $\{B_{ijk}\}$ be symmetric and trace-free grid functions satisfying homogeneous Dirichlet boundary conditions,~\eqref{eq:dirchlet}. Then
	\begin{equation*}
	\innprod{A}{\alpha_h(B)}_h = -2\innprod{\Div_h A}{\Div_h B}_h.
	\end{equation*}
\end{lemma}
\begin{proof}
	We compute (denoting $\alpha_{ijk}:= \alpha_h(B_{ijk})$)
	\begin{align*}
	\innprod{A}{\alpha_h(B)}_h
	&= h^3\sum_{i,j,k=0}^{N+1} \sum_{w,s=1}^3 \left(A_{ijk}\right)_{ws} \left( \alpha_{ijk}\right)_{ws} \\
	&= h^3\sum_{i,j,k=0}^{N+1} \Big(\sum_{w,s=1}^3 \sum_{\beta=1}^3\left[ \left(A_{ijk}\right)_{ws} D_w^c D_\beta^c \left(B_{ijk}\right)_{s\beta} + \left(A_{ijk}\right)_{ws}D_s^c D_\beta^c \left(B_{ijk}\right)_{w\beta} \right]\\
	&\hphantom{= h^d\sum_{i,j,k=1}^N \Big(}-\frac{2}{3} \sum_{w,s=1}^3 \sum_{\beta,\gamma=1}^3 \left(A_{ijk}\right)_{ws} D_\beta^c D_\gamma^c \left(B_{ijk}\right)_{\beta \gamma} \delta_{ws}\Big) 
	\end{align*}
	Focusing on the last term of the inner sum and using that $\delta_{ws} = 1 \Leftrightarrow w=s$
	\begin{align}\label{eq:tracefreealpha}
	\begin{split}
	\sum_{w,s=1}^3 \sum_{\beta,\gamma=1}^3 \left(A_{ijk}\right)_{ws} D_\beta^c D_\gamma^c \left(B_{ijk}\right)_{\beta \gamma} \delta_{ws} &= \sum_{\beta,\gamma=1}^3 \sum_{w=1}^3 \left(A_{ijk}\right)_{ww} D_\beta^c D_\gamma^c \left(B_{ijk}\right)_{\beta \gamma}\\
	&= \sum_{\beta,\gamma=1}^3 D_\beta^c D_\gamma^c \left(B_{ijk}\right)_{\beta \gamma} \sum_{w=1}^3 \left(A_{ijk}\right)_{ww} \\
	&= 0,
	\end{split}
	\end{align}
	where in the last equality we used that $A_{ijk}$ is trace-free. 
	For the first two terms, we use Lemma~\ref{sbp} and the symmetry assumption to obtain
	\begin{align*}
	\begin{split}
	\innprod{A}{\alpha_h(B)}_h &= -h^3\sum_{i,j,k=0}^{N+1} \sum_{w,s,\beta=1}^3\left[ D_w^c \left(A_{ijk}\right)_{ws} D_\beta^c \left(B_{ijk}\right)_{s\beta} + D_s^c \left(A_{ijk}\right)_{ws} D_\beta^c \left(B_{ijk}\right)_{w\beta}\right]  \\
	&=  -h^3\sum_{i,j,k=0}^{N+1} \sum_{w,s,\beta=1}^3\left[ D_w^c \left(A_{ijk}\right)_{sw} D_\beta^c \left(B_{ijk}\right)_{s\beta} + D_s^c \left(A_{ijk}\right)_{ws} D_\beta^c \left(B_{ijk}\right)_{w\beta}\right]  \\
	& = -2\,\innprod{\Div_h A}{\Div_h B}_h,
	\end{split}
	\end{align*}
	which completes the proof.
\end{proof}

Next, we show that the scheme preserves the trace-free and symmetry property of $Q$. To this end, we rewrite the scheme~\eqref{3d_mc_nf_d_scheme} as
$\mathbb{A}(Q^{n+1}) = \mathbb{F}(Q^n)$ where 
\begin{equation}\label{eq:AF}
\begin{cases}
\mathbb{A}(Q^{n+1}_{ijk}) &= \frac{Q^{n+1}_{ijk}}{\Delta t} - \frac{ML_1}{2} \Delta_h Q^{n+1}_{ijk} + \frac{M}{2}\left( \overline{P}^{\hlfstp}_{ijk} : Q^{n+1}_{ijk} \right) \overline{P}^{\hlfstp}_{ijk} - M \frac{L_2+L_3}{4} \alpha_h(Q^{n+1}_{ijk}) \\
\mathbb{F}(Q^n_{ijk}) &= \frac{Q^n_{ijk}}{\Delta t} + \frac{ML_1}{2} \Delta_h Q^n_{ijk} + \frac{M}{2} \left( \overline{P}^{\hlfstp}_{ijk} : Q^n_{ijk} \right) \overline{P}^{\hlfstp}_{ijk} - Mr^n_{ijk} \overline{P}^{\hlfstp}_{ijk} + M \frac{L_2+L_3}{4} \alpha_h(Q^n_{ijk})
\end{cases}
\end{equation}
	where we have used that 
\begin{equation*}
r^{\hlfstp}_{ijk} = \frac{1}{2} \overline{P}^{\hlfstp}_{ijk} : Q^{n+1}_{ijk} - \frac{1}{2} \overline{P}^{\hlfstp}_{ijk} : Q^n_{ijk} + r^n_{ijk}.
\end{equation*}
\begin{proposition}
	If $Q^{n}$ and  $Q^{n-1}$ are trace-free and symmetric, then $Q^{n+1}$ computed by the scheme~\eqref{3d_mc_nf_d_scheme} is also trace-free and symmetric.
\end{proposition}
\begin{proof}
	
	Since we assume that $Q^n_{ijk}, Q^{n-1}_{ijk}$ are trace-free, it follows that also $\overline{P}^{\hlfstp}_{ijk}$ is trace-free.
	 Moreover, $\alpha_h(Q)$ is trace-free without any assumptions on $Q$.
Hence, we find that $\tr(\mathbb{F}(Q^n_{ijk})) = 0$.
But since $\mathbb{A}(Q^{n+1}_{ijk}) = \mathbb{F}(Q^n_{ijk})$ then we must have $\tr(\mathbb{A}(Q^{n+1}_{ijk})) = 0$.
Hence,
	\begin{equation*}
	\tr(\mathbb{A}(Q^{n+1}_{ijk})) = \frac{\tr(Q^{n+1}_{ijk})}{\Delta t} - \frac{ML_1}{4} \Delta_h \tr(Q^{n+1}_{ijk}) = 0
	\end{equation*}
	Taking the inner product of this with $\tr(Q^{n+1}_{ijk})$ we then find 
	\begin{align*}
	\frac{\norm{\tr(Q^{n+1})}^2_h}{\Delta t} - \frac{ML_1}{4} \innprod{\Delta_h \tr(Q^{n+1})}{\tr(Q^{n+1})}_h &= 0.
	\end{align*}
	We use Lemma~\ref{sbp} for the second term:
	\begin{align*}
	\begin{split}
	\innprod{\Delta_h \tr(Q^{n+1})}{\tr(Q^{n+1})}_h & = h^3 \sum_{i,j,k=0}^{N+1}\sum_{\ell=1}^3  \left( D^+_\ell D^-_\ell \sum_{w=1}^3 (Q^{n+1}_{ijk})_{ww}\right)  \sum_{s=1}^3 (Q^{n+1}_{ijk})_{ss}\\
	& = -h^3 \sum_{i,j,k=0}^{N+1} \sum_{\ell=1}^3 \left(  D^-_\ell \sum_{w=1}^3 (Q^{n+1}_{ijk})_{ww}\right) \left(D^-_\ell \sum_{s=1}^3 (Q^{n+1}_{ijk})_{ss}\right)\\
	& = -\norm{\nabla_h \tr(Q^{n+1})}_{h}^2.
	\end{split}
	\end{align*}
		Thus we must have that $\tr(Q^{n+1}_{ijk}) = 0$ for all $i,j,k= 1,\dots, N$ and we see that the trace-free condition is preserved.
		
		For the symmetry, we notice that if $Q^n_{ijk}$ and $Q^{n-1}_{ijk}$ are symmetric, then also $\overline{P}^{n+\frac{1}{2}}_{ijk}$ and $\alpha_h(Q^n_{ijk})$ are symmetric. Hence $\mathbb{F}(Q^n_{ijk}) =(\mathbb{F}(Q^n_{ijk}) )^\top$ and therefore $\mathbb{A}(Q^{n+1}_{ijk})=(\mathbb{A}(Q^{n+1}_{ijk}))^\top$. Denoting $V^{n+1}_{ijk}:=Q^{n+1}_{ijk}-(Q^{n+1}_{ijk})^\top$, this implies
	\begin{equation*}
		 \frac{V^{n+1}_{ijk}}{\Delta t} - \frac{ML_1}{2} \Delta_h V^{n+1}_{ijk}-M\frac{L_2+L_3}{4}\alpha_h(V_{ijk}^{n+1}) = 0.
	\end{equation*}	
		Note that $V^{n+1}_{ijk}$ is skew-symmetric and trace-free. 		We take the inner product with $V^{n+1}_{ijk}$ and obtain
		\begin{equation*}
		0=	\frac{\norm{V^{n+1}}^2_h}{\Delta t} - \frac{ML_1}{2} \innprod{\Delta_h V^{n+1}}{V^{n+1}}_h -M\frac{L_2+L_3}{4}\innprod{\alpha_h(V^{n+1})}{V^{n+1}}_h.
		\end{equation*}
		Using Lemma~\ref{lem:sbplaplace}, this can be rewritten as
			\begin{equation}\label{eq:forfaen}
		0=	\frac{\norm{V^{n+1}}^2_h}{\Delta t} + \frac{ML_1}{2} \norm{\Grad_h V^{n+1}}_h^2-M\frac{L_2+L_3}{2}\innprod{\alpha_h(V^{n+1})}{V^{n+1}}_h.
		\end{equation}
		The term involving $\alpha$ on the right hand side is
			\begin{align*}
		\innprod{V^{n+1}}{\alpha_h(V^{n+1})}_h
		&= h^3\sum_{i,j,k=0}^{N+1} \Big(\sum_{w,s=1}^3 \sum_{\beta=1}^3\left[ \left(V^{n+1}_{ijk}\right)_{ws} D_w^c D_\beta^c \left(V^{n+1}_{ijk}\right)_{s\beta} + \left(V^{n+1}_{ijk}\right)_{ws}D_s^c D_\beta^c \left(V^{n+1}_{ijk}\right)_{w\beta} \right]\\
		&\hphantom{= h^d\sum_{i,j,k=1}^N \Big(}-\frac{2}{3} \sum_{w,s=1}^3 \sum_{\beta,\gamma=1}^3 \left(V^{n+1}_{ijk}\right)_{ws} D_\beta^c D_\gamma^c \left(V^{n+1}_{ijk}\right)_{\beta \gamma} \delta_{ws}\Big) \\
		& = h^3\sum_{i,j,k=0}^{N+1} \sum_{w,s=1}^3 \sum_{\beta=1}^3\left[ \left(V^{n+1}_{ijk}\right)_{ws} D_w^c D_\beta^c \left(V^{n+1}_{ijk}\right)_{s\beta} + \left(V^{n+1}_{ijk}\right)_{ws}D_s^c D_\beta^c \left(V^{n+1}_{ijk}\right)_{w\beta} \right],
		\end{align*}
		using that $V^{n+1}$ is trace-free, as in~\eqref{eq:tracefreealpha} (replacing $A$ and $B$ by $V^{n+1}$). Using Lemma~\ref{sbp} and the skew-symmetry of $V^{n+1}$, the remaining terms are:
			\begin{align*}
		\innprod{V^{n+1}}{\alpha_h(V^{n+1})}_h	& =- h^3\sum_{i,j,k=0}^{N+1} \sum_{w,s=1}^3 \sum_{\beta=1}^3\left[ D_w^c\left(V^{n+1}_{ijk}\right)_{ws}  D_\beta^c \left(V^{n+1}_{ijk}\right)_{s\beta} + D_s^c\left(V^{n+1}_{ijk}\right)_{ws} D_\beta^c \left(V^{n+1}_{ijk}\right)_{w\beta} \right]\\
		& =  h^3\sum_{i,j,k=0}^{N+1} \sum_{w,s=1}^3 \sum_{\beta=1}^3\left[ D_w^c\left(V^{n+1}_{ijk}\right)_{sw}  D_\beta^c \left(V^{n+1}_{ijk}\right)_{s\beta} - D_s^c\left(V^{n+1}_{ijk}\right)_{ws} D_\beta^c \left(V^{n+1}_{ijk}\right)_{w\beta} \right]=0
		\end{align*}
		Plugging this into~\eqref{eq:forfaen}, we see that $V^{n+1}_{ijk}=0$ for all $i,j,k$.		
	\end{proof}
The next theorem guarantees the existence of a unique solution of the system of equations~\eqref{3d_mc_nf_d_scheme} (or~\eqref{eq:AF}).
\begin{theorem}
	The operator $\mathbb{A}$ is symmetric and positive definite for grid functions that are symmetric and trace-free. 
\end{theorem}
\begin{proof}
Let $Q^1=(Q_{ijk}^1)_{ijk}$ and $Q^2=(Q_{ijk}^2)_{ijk}$, then 
\begin{align*} 
\begin{split}
\innprod{\mathbb{A}(Q^1)}{Q^2}_h =& \frac{1}{\Delta t} \innprod{Q^1}{Q^2}_h - \frac{ML_1}{2} \innprod{\Delta_h Q^1}{Q^2}_h + \frac{M}{2} \innprod{\overline{P}^{\hlfstp} : Q^1} {\overline{P}^{\hlfstp} : Q^2}_h\\
& - M \frac{L_2+L_3}{4} \innprod{\alpha_h(Q^1)}{Q^2}_h
\end{split}
\end{align*}
By Lemmas~\ref{lem:sbplaplace} and~\ref{lem:sbpalpha}, we have
\begin{equation*}
\innprod{\Delta_h Q^1}{Q^2}_h = -\innprod{\Grad_hQ^1}{\Grad_h Q^2}_h,\quad \innprod{\alpha_h( Q^1)}{Q^2}_h = -2\innprod{\Div_hQ^1}{\Div_h Q^2}_h.
\end{equation*}
Therefore,
\begin{align*}
\begin{split}
\innprod{\mathbb{A}(Q^1)}{Q^2}_h =& \frac{1}{\Delta t} \innprod{Q^1}{Q^2}_h + \frac{ML_1}{2} \innprod{\Grad_h Q^1}{\Grad_h Q^2}_h + \frac{M}{2} \innprod{\overline{P}^{\hlfstp} : Q^1} {\overline{P}^{\hlfstp} : Q^2}_h\\
& +  \frac{M\,(L_2+L_3)}{2} \,\innprod{\Div^-_h Q^1}{\Div^-_h Q^2}_h,
\end{split}
\end{align*}
which we see is symmetric in $Q^1$ and $Q^2$. Moreover, 
\begin{align*}
\begin{split}
\innprod{\mathbb{A}(Q)}{Q}_h =& \frac{1}{\Delta t} \norm{Q}^2_h + \frac{ML_1}{2} \norm{\Grad_h Q}_h^2 + \frac{M}{2} \norm{\overline{P}^{\hlfstp} : Q}_h^2  +  \frac{M\,(L_2+L_3)}{2} \norm{\Div_h Q}^2_h\geq 0,
\end{split}
\end{align*}
with equality if and only if $Q\equiv 0$.
\end{proof}

By the previous two results, we conclude there is a unique solution $Q^{n+1}$ to $\mathbb{A}(Q^{n+1}) = \mathbb{F}(Q^n)$ that is trace-free and symmetric. 

Next, we will prove an energy estimate for the scheme:
\begin{theorem}\label{thm:energyestimate}
Define the energy 
\begin{equation}\label{eq:energydisc}
E^n = \frac{L_1}{2} \norm{\nabla_h Q^{n}}_{h}^2 + \frac{L_2+L_3}{2} \norm{\Div_h Q^{n}}_h^2 + \frac{1}{2} \norm{r^{n}}_h^2
\end{equation}
where $Q^n, r^n$ solve \eqref{3d_mc_nf_d_scheme}. 
Then 
\begin{equation*}
E^{n+1} - E^n = -\Delta t M \norm{H^{\hlfstp}}_h^2.
\end{equation*}
\end{theorem}

\begin{proof} We take the inner product of the first equation in \eqref{3d_mc_nf_d_scheme} with $-\Delta t H^{\hlfstp}_{ijk}$, multiply by $h^d$ and sum over all grid points, 
\begin{equation} \label{1stequation}
- \innprod{Q^{n+1}-Q^n}{H^{\hlfstp}}_h = - \Delta t\,M \norm{H^{\hlfstp}}_{h}^2
\end{equation}
Taking the inner product of the second equation with $r^{\hlfstp}_{ijk}$, multiplying by $h^d$, and summing over all grid points, gives 
\begin{align} \label{2ndEquation}
\begin{split}
\innprod{r^{n+1}-r^n}{r^{\hlfstp}}_h &= \innprod{r^{\hlfstp} \overline{P}^{\hlfstp}}{Q^{n+1}-Q^n}_h \\
\Longleftrightarrow\qquad\frac{1}{2} \norm{r^{n+1}}_{h}^2 - \frac{1}{2}\norm{r^n}_{h}^2 &= \innprod{r^{\hlfstp} \overline{P}^{\hlfstp}}{Q^{n+1}-Q^n}_h
\end{split}
\end{align}
Next we shall work with the term $\innprod{Q^{n+1}-Q^n}{H^{\hlfstp}}_h$. 
\begin{align} \label{Q_R_innprod}
\begin{split}
\innprod{Q^{n+1}-Q^n}{H^{\hlfstp}}_h =& \innprod{Q^{n+1}-Q^n}{L_1 \Delta_h Q^{\hlfstp} - q^{\hlfstp}\overline{P}^{\hlfstp} + \frac{L_2+L_3}{2} \alpha^{\hlfstp}}_h \\
=& \frac{L_1}{2} \innprod{Q^{n+1} - Q^n}{\Delta_h Q^{n+1}+\Delta_h Q^n}_h - \innprod{Q^{n+1}-Q^n}{q^{\hlfstp}\overline{P}^{\hlfstp}}_h \\
&+ \frac{L_2+L_3}{4} \innprod{Q^{n+1}-Q^n}{\alpha^{n+1}+\alpha^n}_h 
\end{split}
\end{align}
We shall deal with these terms individually. Using bilinearity, we have 
\begin{equation} \label{expanded1}
\innprod{Q^{n+1} - Q^n}{\Delta_h Q^{n+1}+\Delta_h Q^n}_h = \innprod{Q^{n+1}}{\Delta_h Q^{n+1}}_h + \innprod{Q^{n+1}}{\Delta_h Q^n}_h - \innprod{Q^n}{\Delta_h Q^{n+1}}_h - \innprod{Q^n}{\Delta_h Q^n}_h. 
\end{equation}
We shall focus on the first term of the right hand side. Using Lemma~\ref{lem:sbplaplace} and the boundary conditions, we have
\begin{equation*}
   \langle Q^{n+1}, \Delta_{h} Q^{n+1}\rangle_{h}
        = -\lVert \nabla_{h} Q^{n+1}\rVert_{h}^{2} 
\end{equation*} 
Similarly, we see that $\innprod{Q^n}{ \Delta_h Q^n}_h = -\norm{\nabla_h Q^n}_{h}^2$, $\innprod{Q^n}{ \Delta_h Q^{n+1}}_h = -\innprod{\Grad_h Q^n}{ \Grad_h Q^{n+1}}_h$ and $\innprod{Q^{n+1}}{ \Delta_h Q^{n}}_h = -\innprod{\Grad_h Q^{n+1}}{ \Grad_h Q^{n}}_h$.
Hence putting all these results into \eqref{expanded1} we find 
\begin{align*}
\innprod{Q^{n+1} - Q^n}{\Delta_h Q^{n+1}+\Delta_h Q^n}_h &= - \norm{\nabla_h Q^{n+1}}_{h}^2 - \innprod{\nabla_h Q^{n+1}}{\nabla_h Q^n}_h + \innprod{\nabla_h Q^n}{\nabla_h Q^{n+1}}_h + \norm{\nabla_h Q^n}_{h}^2 \\
&= - \norm{\nabla_h Q^{n+1}}_{h}^2  + \norm{\nabla_h Q^n}_{h}^2
\end{align*}
Next we consider the term $\innprod{Q^{n+1}-Q^n}{\alpha^{n+1}+ \alpha^n}_h$ in \eqref{Q_R_innprod} :
\begin{equation*}
\innprod{Q^{n+1}-Q^n}{\alpha^{n+1}+ \alpha^n}_h = \innprod{Q^{n+1}}{\alpha^{n+1}}_h + \innprod{Q^{n+1}}{\alpha^n}_h - \innprod{Q^n}{\alpha^{n+1}}_h - \innprod{Q^n}{\alpha^n}_h.
\end{equation*}
Using Lemma~\ref{lem:sbpalpha} for each of these terms, we obtain
\begin{equation*}
\innprod{Q^{n+1}-Q^n}{\alpha^{n+1}+\alpha^n}_h = -2 \norm{\Div_h Q^{n+1}}_h^2 + 2 \norm{\Div_h Q^n}_h^2.
\end{equation*}
Overall we have shown that
\begin{align*}
\innprod{Q^{n+1}-Q^n}{H^{\hlfstp}}_h = &\frac{L_1}{2}\left(-\norm{\nabla_h Q^{n+1}}_{h}^2 + \norm{\nabla_h Q^n}_{h}^2 \right) + \frac{L_2+L_3}{2} \left(-\norm{\Div_h Q^{n+1}}_h^2 + \norm{\Div_h Q^n}_h^2\right)  \\
&- \innprod{Q^{n+1}-Q^n}{r^{\hlfstp}\overline{P}^{\hlfstp}}_h.
\end{align*}
Combining this with \eqref{1stequation} and \eqref{2ndEquation}, we obtain
\begin{multline*}
\frac{L_1}{2} \left( \norm{\nabla_h Q^{n+1}}_{h}^2 - \norm{\nabla_h Q^n}_{h}^2\right)
 + \frac{L_2+L_3}{2} \left( \norm{\Div_h Q^{n+1}}_h^2 - \norm{\Div_h Q^n}_h^2\right)\\
\qquad  + \frac{1}{2}\left(\norm{r^{n+1}}_h^2 - \norm{r^n}_h^2 \right) = -\Delta t\,M\norm{H^{\hlfstp}}^2_h.
\end{multline*}
\end{proof}

Based on the energy estimate, Theorem~\ref{thm:energyestimate}, we can derive further stability bounds on the approximations $\{Q_{ijk}^n\}$ and $\{r_{ijk}^n\}$.
Specifically, it follows from the bound on $\{H^{n+\frac{1}{2}}\}$ that $D_t^+Q_{ijk}^n=MH^{n+\frac{1}{2}}_{ijk}$ is bounded:
\begin{corollary}\label{cor:timeder}
	We have
	\begin{equation*}
	\Delta t\sum_{n=0}^{N_T-1}\norm{D_t^+ Q^n}_{h}^2\leq E^0,
	\end{equation*}
	where $N_T$ is such that $T=N_T \Delta t$.
\end{corollary}
Using this corollary and the energy estimate, we can also derive a uniform (in $\Delta t$ and $h$) bound on $\norm{Q^{n}}_h$.
\begin{lemma}\label{lem:l2bound}
	The following estimate holds for any $\Delta t, h>0$:
		\begin{equation*}
		\norm{Q^{m}}_h\leq T^\frac{1}{2}\left(\Delta t\sum_{n=0}^{m-1}\norm{D_t^+ Q^n}_{h}^2\right)^{\frac{1}{2}}+\|Q^0\|_h<\infty,
	\end{equation*}
	for any $0\leq m\leq N_T$ where $N_T$ is such that $N_T\Delta t =T$.
\end{lemma} 
\begin{proof}
	Note that 
	\begin{align*}
	\norm{Q^{n+1}}_h^2-\norm{Q^{n}}_h^2&=h^3\sum_{i,j,k=0}^{N+1} |Q_{ijk}^{n+1}|_F^2-h^3\sum_{i,j,k=0}^{N+1}|Q_{ijk}^{n}|_F^2\\
	&=h^3\sum_{i,j,k=0}^{N+1}(Q_{ijk}^{n+1}+Q_{ijk}^{n}):(Q_{ijk}^{n+1}-Q_{ijk}^{n})\\
	&\leq\left[h^3\sum_{i,j,k=0}^{N+1} |Q_{ijk}^{n+1}+Q_{ijk}^{n}|_F^2\right]^\frac{1}{2}\,\left[h^3\sum_{i,j,k=0}^{N+1} |Q_{ijk}^{n+1}-Q_{ijk}^{n}|_F^2\right]^\frac{1}{2}\\
	&\leq\left[\left(h^3\sum_{i,j,k=0}^{N+1} |Q_{ijk}^{n+1}|_F^2\right)^\frac{1}{2}+\left(h^3\sum_{i,j,k=0}^{N+1}|Q_{ijk}^{n}|_F^2\right)^\frac{1}{2}\right]\,\left[h^3\sum_{i,j,k=0}^{N+1}\left|\frac{Q_{ijk}^{n+1}-Q_{ijk}^{n}}{\Delta t}\right|_F^2\right]^\frac{1}{2}\,\Delta t\\
	&=\left(\norm{Q^{n+1}}_h+\norm{Q^{n}}_h\right)\,\|D_t^+Q^n\|_h\,\Delta t,
	\end{align*}
	and so 
	\begin{equation*}
	 \norm{Q^{n+1}}_h-\norm{Q^{n}}_h \leq \|D_t^+Q^n\|_h\,\Delta t,
	\end{equation*}
	for any $0\leq n\leq N_T$.
	Summing over $n$ on both sides, we obtain that for any $0\leq m\leq N_T$,
	\begin{equation*} 
	\begin{aligned}
	\norm{Q^{m}}_h\leq\sum_{n=0}^{m-1}\|D_t^+Q^n\|_h\,\Delta t+\|Q^0\|_h&\leq \left(\sum_{n=0}^{m-1}\Delta t\right)^\frac{1}{2}\,\left(\sum_{n=0}^{m-1}\|D_t^+Q^n\|_h^2\,\Delta t\right)^\frac{1}{2}+\|Q^0\|_h\\
	&\leq T^{\frac{1}{2}}\,\left(\sum_{n=0}^{m-1}\|D_t^+Q^n\|_h^2\,\Delta t\right)^\frac{1}{2}+\|Q^0\|_h.
	\end{aligned}
	\end{equation*}	
\end{proof}
\subsection{Lipschitz continuity of $P(Q)$}
In order to derive a stability bound on $\{D_t^+ r_{ijk}^n\}$, we need an auxiliary result, which is the Lipschitz continuity of $P(Q)$.  
Recall that we can write $P(Q)$ as
$P(Q)=\frac{S(Q)}{r(Q)}$ where $S$ and $r$ have been defined in~\eqref{eq:defS} and~\eqref{eq:defr}. 
Note that we can express the Frobenius norm as
\begin{equation*}
|Q|_F=\sqrt{\tr(Q^2)}=\sqrt{\sum_{i=1}^d\lambda_i^2},
\end{equation*}
where $\lambda_i$ is the $i$th eigenvalue of matrix $Q$.

We start with a few preliminary lemmas. First, note that $r(Q)$ is bounded from below by some constant $A>0$ (see also~\cite[Theorem 2.1]{ZhaoYang2017}). We will also need an upper bound. Since $c>0$, there exists constant $K_1>0$ such that for any $Q$ for which $|Q|_F\geq K_1$,
\begin{equation*}
\left|{a}\tr(Q^2)-\frac{2b}{3}\tr(Q^3)+2A_0\right|\leq\frac{c}{4}\tr^2(Q^2).
\end{equation*} 
Then 
\begin{equation*}
r(Q)\geq\sqrt{ \frac{{c}}{2}\text{tr}^2(Q^2)-\left|{a}\text{tr}(Q^2)-\frac{2b}{3}\text{tr}(Q^3)+2A_0\right|}\geq\sqrt{\frac{{c}}{2}\tr^2(Q^2)-\frac{{c}}{4}\tr^2(Q^2)}=\frac{\sqrt{c}}{2}\,\tr(Q^2),
\end{equation*}
and 
\begin{equation*}
r(Q)\leq\sqrt{ \frac{{c}}{2}\text{tr}^2(Q^2)+\left|{a}\tr(Q^2)-\frac{2b}{3}\text{tr}(Q^3)+2A_0\right|}\leq\sqrt{\frac{{c}}{2}\text{tr}^2(Q^2)+\frac{{c}}{4}\tr^2(Q^2)}\leq{\sqrt{c}}\tr(Q^2).
\end{equation*}
So whenever $|Q|_F\geq K_1$, we can bound $r(Q)$ by 
\begin{equation}\label{rqestimate1}
\frac{\sqrt{c}}{2}\,|Q|^2_F=\frac{\sqrt{c}}{2}\,\text{tr}(Q^2)\leq r(Q)\leq \sqrt{c}\tr(Q^2)=\sqrt{c}\,|Q|_F^2.
\end{equation}
On the other hand, when $Q$ is bounded by constant $K_1$, we have 
\begin{equation*} \begin{aligned}
r(Q)&\leq\sqrt{2\,\left[\frac{|a|}{2}\tr(Q^2)+\frac{|b|}{3}\left|\tr(Q^3)\right|+\frac{c}{4}\tr^2(Q^2)+A_0\right]}\\
&\leq\sqrt{2\,\left(\frac{|a|}{2} K^2_1+\frac{|b|}{3} K_1^3+\frac{c}{4} K_1^4+A_0\right)}\triangleq K_2,
\end{aligned}
\end{equation*}
where we have used the fact that%
\begin{equation*}
\tr(Q^4)=\sum_{i=1}^d\lambda_i^4\leq \left(\sum_{i=1}^d\lambda_i^2\right)^2\leq\tr^2(Q^2),
\end{equation*}
and then
\begin{equation*}
|\tr(Q^3)|\leq\tr^\frac{1}{2}(Q^4)\tr^\frac{1}{2}(Q^2)\leq\tr^\frac{3}{2}(Q^2)=K_1^3.
\end{equation*}
Combining the two results, we obtain that
\begin{equation}\label{rqestimate2}
r(Q)\leq K_2+\sqrt{c}|Q|^2_F,
\end{equation}
for some constant $K_2>0$. 
This bound will be used subsequently. The following lemmas are important steps towards our Lipschitz estimate for $P(Q)$.
\begin{lemma}\label{QoverrQ}
	For any $Q$, there exist constants $C_1$, $C_2$ and $C_3$ such that 
	\begin{equation*}
	\frac{1}{r(Q)}\leq C_1,\quad
	\frac{|Q|_F}{r(Q)}\leq C_2,\quad
	\frac{|Q|^2_F}{r(Q)}\leq C_3.
	\end{equation*}
\end{lemma}
\begin{proof}
	The first estimate follows from the fact that $r(Q)$ is bounded from below.
	For the third estimate, we split it into two cases. When $|Q|_F\leq K_1$, we have
	\begin{equation*}
	\frac{|Q|^2_F}{r(Q)}\leq \frac{K^2_1}{A}.
	\end{equation*}
	
	When $|Q|_F\geq K_1$, by \eqref{rqestimate1}, we know that
	\begin{equation*}
	\frac{|Q|^2_F}{r(Q)}\leq\frac{|Q|^2_F}{\frac{\sqrt{c}}{2}|Q|^2_F}\leq\frac{2}{\sqrt{c}}.
	\end{equation*}
	Define $C_3\triangleq \max\left\{\frac{K^2_1}{A},\frac{2}{\sqrt{c}}\right\}$, then $\frac{|Q|^2_F}{r(Q)}\leq C_3$.
	To prove the second estimate, we note that if $|Q|_F\leq 1$, then 
	\begin{equation*}
	\frac{|Q|_F}{r(Q)}\leq \frac{1}{r(Q)}\leq C_1.
	\end{equation*}
	Else, we have that 
	\begin{equation*}
	\frac{|Q|_F}{r(Q)}\leq\frac{|Q|_F^2}{r(Q)}\leq C_3.
	\end{equation*}
	Defining $C_2\triangleq\max\{C_1,C_3\}$, we obtain $\frac{|Q|_F}{r(Q)}\leq C_2$ which completes the proof of the lemma.
\end{proof}
\begin{lem}\label{lem:SQbound}
	For any matrix $Q$, $\left|\frac{S(Q)}{r(Q)^{\frac{3}{2}}}\right|_F$ is uniformly bounded.
\end{lem}
\begin{proof}
	Note that
	\begin{equation*}
	\left|S(Q)\right|_F=\left|aQ-b\left[Q^2-\frac{1}{d}\text{tr}(Q^2)I\right]+c\text{tr}(Q^2)Q\right|_F\leq|a| |Q|_F+|b| |Q^2|_F+\frac{|b|}{d}\,|Q|^2_F+c|Q|_F^2|Q|_F.
	\end{equation*}
	Since
	\begin{equation*}
	|Q^2|_F=\sqrt{\tr(Q^4)}\leq\sqrt{\tr^2(Q^2)}=\tr(Q^2)=|Q|_F^2,
	\end{equation*}
	we obtain
	\begin{equation*}
	|S(Q)|_F\leq|a||Q|_F+\frac{(d+1)|b|}{d}|Q|^2_F+c\,|Q|_F^3.
	\end{equation*} 
	Then from Lemma~\ref{QoverrQ}, we obtain that
	\begin{align*}
	\left|\frac{S(Q)}{r(Q)^{\frac{3}{2}}}\right|_F&\leq\frac{|a|\,|Q|_F+\frac{(d+1)|b|}{d}\,|Q|^2_F+c\,|Q|_F^3}{r(Q)^\frac{3}{2}}\\
	&\leq \frac{|a|}{r(Q)^\frac{1}{2}}\frac{|Q|_F}{r(Q)}+\frac{\frac{(d+1)|b|}{d}}{r(Q)^\frac{1}{2}}\frac{|Q|^2_F}{r(Q)}+c\left(\frac{|Q|^2_F}{r(Q)}\right)^\frac{3}{2}\\
	&\leq\frac{|a|}{A^\frac{1}{2}}C_2+\frac{(d+1)|b|}{dA^\frac{1}{2}}C_3+c\,C_3^\frac{3}{2}\triangleq K_3,
	\end{align*}
	which proves the lemma.
\end{proof}
Now we are in a position to prove that $P(Q)$ is Lipschitz continuous with respect to the Frobenius norm.
\begin{thm}\label{thm:PLipschitz}
	There exists a constant $L>0$ such that for any matrices $Q, \delta Q\in \mathbb{R}^{3\times 3}$,
	\begin{equation*}
	|P(Q+\delta Q)-P(Q)|_F\leq L|\delta Q|_F.
	\end{equation*}
\end{thm}
\begin{proof}
	We will split the proof into two cases. 
	
	\medskip
	
	{\bf Case 1: } $\delta Q$ is so large such that $|\delta Q|_F\geq 2|Q|_F$ and $|\delta Q|_F\geq \max\{2K_1,\,K_3\sqrt{K_2}\}\triangleq G.$ In this case, we can see that
	\begin{equation*}
	|\delta Q+Q|_F\geq|\delta Q|_F-|Q|_F\geq\frac{1}{2}\,|\delta Q|_F\geq K_1,
	\end{equation*}
	therefore, by (\ref{rqestimate1}), we have
	\begin{equation}\label{eq:rdeltaestimate}
	r(Q+\delta Q)\leq\sqrt{c}|Q+\delta Q|_F\leq \sqrt{c}(|Q|_F+|\delta Q|_F)\leq\frac{3\sqrt{c}}{2}|\delta Q|_F.
	\end{equation}
	We use this to compute the difference between $P(Q+\delta Q)$ and $P(Q)$,
	\begin{align*}
	|P(Q+\delta Q)-P(Q)|_F&\leq|P(Q+\delta Q)|_F+|P(Q)|_F\\
	&=\left|\frac{S(Q+\delta Q)}{r(Q+\delta Q)}\right|_F+\left|\frac{S(Q)}{r(Q)}\right|_F\\
	&=\left|\frac{S(Q+\delta Q)}{r(Q+\delta Q)^\frac{3}{2}}\right|_F\sqrt{r(Q+\delta Q)}+\left|\frac{S(Q)}{r(Q)^\frac{3}{2}}\right|_F\sqrt{r(Q)}\\
	&\stackrel{\text{Lem } \ref{QoverrQ} }{\leq} K_3\sqrt{r(Q+\delta Q)}+K_3\sqrt{r(Q)}\\
	&\stackrel{\eqref{rqestimate2},\eqref{eq:rdeltaestimate}}{\leq} \frac{3K_3\sqrt{c}}{2}|\delta Q|_F+K_3(\sqrt{K_2}+c^\frac{1}{4}|Q|_F)\\
	&\leq\frac{3K_3\sqrt{c}}{2}|\delta Q|_F+|\delta Q|_F+\frac{K_3c^\frac{1}{4}}{2}|\delta Q|_F\\
	&= \left(\frac{3K_3\sqrt{c}}{2}+1+\frac{K_3c^\frac{1}{4}}{2}\right)|\delta Q|_F,
	\end{align*}
	which proves the result in this case.
	\smallskip
	
	{\bf Case 2: } $|\delta Q|_F\leq 2|Q|_F$ or $|\delta Q|_F\leq G$.
	
	In this case, we write the difference of $P(Q+\delta Q)$ and $P(Q)$ as
	\begin{align*}
	|P(Q+\delta Q)-P(Q)|_F&=\left|\frac{S(Q+\delta Q)}{r(Q+\delta Q)}-\frac{S(Q)}{r(Q)}\right|_F\\
	&=\left|\frac{S(Q+\delta Q)-S(Q)}{r(Q)}\,+\,S(Q+\delta Q)\left(\frac{1}{r(Q+\delta Q)}-\frac{1}{r(Q)}\right)\right|_F\\
	&\leq\underbrace{\left|\frac{S(Q+\delta Q)-S(Q)}{r(Q)}\right|_F}_\Romannum{1}
	+\underbrace{\left|\frac{S(Q+\delta Q)}{r(Q+\delta Q)^{\frac{3}{2}}}\,\,\frac{\sqrt{r(Q+\delta Q)}\left[r(Q+\delta Q)-r(Q)\right]}{r(Q)}\right|_F}_{\Romannum{2}}.
	\end{align*}
	To compute $\Romannum{1}$, we expand $S(Q+\delta Q)$ by plugging in $(Q+\delta Q)$ into (\ref{eq:defS}): 
	\begin{align*}
	S(Q+\delta Q)
	=&S(Q)+a\,\delta Q-b\,(Q\,\delta Q+\delta Q\,Q)-b\,(\delta Q)^2+\frac{2b}{d}\tr\left(Q\delta Q\right)I\\
	&+\frac{b}{d}\tr\left((\delta Q)^2\right)I+c\tr((\delta Q)^2)\delta Q+2c\tr(Q\,\delta Q)Q\\
	&+2c\tr(Q\,\delta Q)\delta Q+c\tr(Q^2)\delta Q+c\tr\left((\delta Q)^2\right)Q.
	\end{align*}
	Then by Lemma~\ref{QoverrQ}, we have
	\begin{align}
	\Romannum{1}\leq& |a|\frac{|\delta Q|_F}{r(Q)}+\frac{2(d+1)|b|}{d}\frac{|Q|_F|\delta Q|_F}{r(Q)}+\frac{(d+1)|b|}{d}\frac{|\delta Q|^2_F}{r(Q)}+c\frac{|\delta Q|_F^3+3|Q|_F^2|\delta Q|_F+3|\delta Q|_F^2|Q|_F}{r(Q)}\notag\\
	\leq& |a|\,C_1\,|\delta Q|_F+\frac{2(d+1)C_2|b|}{d}|\delta Q|_F+\frac{(d+1)|b|}{d}\frac{|\delta Q|_F^2}{r(Q)}+3c\,C_3|\delta Q|_F+c\,\frac{|\delta Q|_F^3}{r(Q)}+3c\frac{|Q|_F|\delta Q|_F^2}{r(Q)}.\label{eq:randomeq}
	\end{align}
	We still need to bound $\frac{|\delta Q|_F^2}{r(Q)}$, $\frac{|\delta Q|_F^3}{r(Q)}$ and $\frac{|Q|_F|\delta Q|_F^2}{r(Q)}$ in terms of $\delta Q$. Based on our assumption in this case, if $|\delta Q|_F\leq 2|Q|_F$, then
	\begin{align*}
	\frac{|\delta Q|_F^2}{r(Q)}\leq &\frac{2|Q|_F\,|\delta Q|_F}{r(Q)}\leq 2C_2\,|\delta Q|_F,\\
	\frac{|\delta Q|_F^3}{r(Q)}\leq &\frac{4|Q|_F^2|\delta Q|_F}{r(Q)}\leq 4C_3\,|\delta Q|_F\\
	\text{and}\quad \frac{|Q|_F|\delta Q|_F^2}{r(Q)}&\leq\frac{2|Q|^2_F|\delta Q|_F}{r(Q)}\leq2C_3|\delta Q|_F.
	\end{align*}
	Hence, plugging this into~\eqref{eq:randomeq}, we obtain
	\begin{equation*}
	\Romannum{1}
	\leq\left(|a|\,C_1+\frac{4(d+1)}{d}|b|\,C_2+13c\,C_3\right)|\delta Q|_F
	\end{equation*}
	On the other hand, if $|\delta Q|_F\leq G,$ then we can bound $\frac{|\delta Q|_F^2}{r(Q)}$, $\frac{|\delta Q|_F^3}{r(Q)}$ and $\frac{|Q|_F|\delta Q|_F^2}{r(Q)}$ by
	\begin{equation*}
	\frac{|\delta Q|_F^2}{r(Q)}\leq G C_1|\delta Q|_F,\quad\frac{|\delta Q|_F^3}{r(Q)}\leq G^2C_1|\delta Q|_F,\quad\frac{|Q|_F|\delta Q|_F^2}{r(Q)}\leq G\,C_2|\delta Q|_F.
	\end{equation*}
	Plugging these into~\eqref{eq:randomeq}, we arrive at
	\begin{align*}
	\Romannum{1}
	\leq\left(|a|\,C_1+\frac{2(d+1)}{d}|b|\,C_2+\frac{d+1}{d}|b|\,G\,C_1+3c\,C_3+c\,G^2\,C_1+3c\,G\,C_2\right)|\delta Q|_F.
	\end{align*}
	Therefore, $\Romannum{1}\leq Z_1 |\delta Q|_F$ for some constant $Z_1$ depending on $C_i$, $i=1,2,3$, $G,a,b$ and $c$.
	To bound term $\Romannum{2}$, note that 
	\begin{align}
	\Romannum{2}&\leq\left|\frac{S(Q+\delta Q)}{r(Q+\delta Q)^{\frac{3}{2}}}\right|_F\,\frac{\sqrt{r(Q+\delta Q)}|r(Q+\delta Q)-r(Q)|}{r(Q)}\notag\\
	&\stackrel{\text{Lem } \ref{lem:SQbound}}{\leq} K_3\,\frac{\sqrt{r(Q+\delta Q)}\,\,|r(Q+\delta Q)-r(Q)|}{r(Q)}\notag\\
	&=K_3\,\frac{\sqrt{r(Q+\delta Q)}\,\,|r(Q+\delta Q)^2-r(Q)^2|}{r(Q)\,\,[r(Q+\delta Q)+r(Q)]}\notag\\
	&\leq K_3\,\frac{\,\,|r(Q+\delta Q)^2-r(Q)^2|}{r(Q)^\frac{3}{2}},\label{eq:IIprelim}
	\end{align}
	where we have used the fact 
	\begin{equation*}
	r(Q+\delta Q)+r(Q)\geq\,2\sqrt{r(Q+\delta Q)\,r(Q)}\geq\,\sqrt{r(Q+\delta Q)\,r(Q)}.
	\end{equation*}
	Expanding $r(Q+\delta Q)$ by plugging in $(Q+\delta Q)$ into (\ref{eq:defr}), we have 
	\begin{equation*}
	\begin{aligned}
	r(Q+\delta Q)^2
	=&r(Q)^2+2a\tr(Q\,\delta Q)+a\tr((\delta Q)^2)-\frac{2b}{3}\tr((\delta Q)^3)-2b\tr(Q^2\,\delta Q)\\
	&-2b\tr(Q\,(\delta Q)^2)+\frac{c}{2}\text{tr}^2((\delta Q)^2)+2c\tr^2(Q\,\delta Q)\\
	&+2c\tr(Q^2)\tr(Q\,\delta Q)+c\tr(Q^2)\tr((\delta Q)^2)+2c\tr((\delta Q)^2)\tr(Q\,\delta Q).
	\end{aligned}
	\end{equation*}
	We plug this into~\eqref{eq:IIprelim},
	\begin{align*}
	\frac{1}{K_3}\,\Romannum{2}\leq&\frac{2
		|a|\,|Q|_F|\delta Q|_F+|a|\,|\delta Q|_F^2+\frac{2|b|}{3}|\delta Q|_F^3+2|b||Q|_F^2|\delta Q|_F+2|b||Q|_F|\delta Q|_F^2}{r(Q)^\frac{3}{2}}\\
	&+\frac{\frac{c}{2}|\delta Q|^4_F+3c|Q|_F^2|\delta Q|_F^2+2c\,|Q|_F^3|\delta Q|_F+2c|\delta Q|^3_F|Q|_F}{r(Q)^\frac{3}{2}}.
	\end{align*}
	In a similar way as for term $\Romannum{1}$, we can find constant $Z_2$ such that
	$\Romannum{2}\leq Z_2|\delta Q|_F$.
	To sum up, if we choose $L=\max\{\frac{3K_3\sqrt{c}}{2}+1+\frac{K_3c^\frac{1}{4}}{2},Z_1,Z_2\}$, then 
	\begin{equation*}
	|P(Q+\delta Q)-P(Q)|_F\leq L|\delta Q|_F,
	\end{equation*}
	for any $Q$ and $\delta Q$.
\end{proof}
Using the Lipschitz continuity of $P(Q)$, it is now easy to prove the following bound on $\{D_t^+r_{ijk}^n\}$:
\begin{lemma}
	\label{lem:rtbound}
	We have
	\begin{equation*}
	\Delta t\sum_{n=0}^m\left(h^3\sum_{i,j,k=1}^{N} |D_t^+ r_{ijk}^n|\right)^2\leq C<\infty,
	\end{equation*}
	for $0\leq m\leq N_T$ where $N_T$ is such that $N_T\Delta t =T$ and $C>0$ is a constant independent of $h$ and $\Delta t$.
\end{lemma}
\begin{proof}
	We take absolute values of the scheme for $r_{ijk}^n$, the second equation in~\eqref{3d_mc_nf_d_scheme} divided by $\Delta t$:
	\begin{equation*}
	\left|D_t^+ r_{ijk}^n\right| = \left|\overline{P}^{n+\frac{1}{2}}_{ijk}:D_t^+ Q_{ijk}^n \right|,
	\end{equation*}
	and sum over $i,j,k=1,\dots N$, then multiply by $h^3$, square and sum over $n=0,\dots, m$ and use H\"older's inequality:
	\begin{equation*}
	\begin{split}
	\sum_{n=0}^m\left(h^3\sum_{i,j,k=1}^{N} \left|D_t^+ r_{ijk}^n\right|\right)^2 &= \sum_{n=0}^m\left(h^3\sum_{i,j,k=1}^{N} \left|\overline{P}^{n+\frac{1}{2}}_{ijk}:D_t^+ Q_{ijk}^n \right|\right)^2\\
&\leq  \sum_{n=0}^m\left(h^3\sum_{i,j,k=1}^{N} \left|\overline{P}^{n+\frac{1}{2}}_{ijk}\right|^2\right)\left(h^3\sum_{i,j,k=1}^N \left|D_t^+ Q_{ijk}^n \right|^2\right).	
	\end{split}
	\end{equation*}
	Next, we use the Lipschitz continuity of $P(Q)$, and then Lemma~\ref{lem:l2bound},
		\begin{equation*}
	\begin{split}
	\sum_{n=0}^m\left(h^3\sum_{i,j,k=1}^{N} \left|D_t^+ r_{ijk}^n\right|\right)^2 &\leq  C h^{6}\sum_{n=0}^m\sum_{i,j,k=1}^{N} \left(\left|Q_{ijk}^n\right|^2+\left|Q_{ijk}^{n-1}\right|^2\right)\sum_{i,j,k=1}^{N} \left|D_t^+ Q_{ijk}^n \right|^2\\
	&\leq C   \sum_{n=0}^m\norm{D_t^+ Q^n}_h^2 \,\max_{0\leq \ell\leq m}\norm{Q^\ell}_h^2\\
	&\stackrel{\text{Lem. } \ref{lem:l2bound}}{\leq } C \sum_{n=0}^m\norm{D_t^+ Q^n}_h^2
	\end{split}\end{equation*}
	Multiplying by $\Delta t$ and using Corollary~\ref{cor:timeder}, we obtain the result.
	
\end{proof}

\section{Convergence of the scheme}
Using the estimates established in the previous section, we proceed to proving convergence of the scheme~\eqref{3d_mc_nf_d_scheme} to a weak solution of~\eqref{eq:reformulation}.  
To do so, we define piecewise constant interpolations of the grid functions $\{Q^n_{ijk}\}$, $\{r^n_{ijk}\}$ and $\{\overline{P}^{n+\frac{1}{2}}_{ijk}\}$,
 \begin{equation}\label{eq:discrete_solution}
Q^n_{h,\Delta t}(x)=\sum_{i,j,k=0}^{N+1}Q_{ijk}^n\,\chi_{C_{ijk}},\quad r^n_{h,\Delta t}(x)=\sum_{i,j,k=0}^{N+1}r_{ijk}^n\,\chi_{C_{ijk}},\quad P^n_{h,\Delta t}(x)=\sum_{i,j,k=0}^{N+1}\overline{P}_{ijk}^{n+\frac{1}{2}}\,\chi_{C_{ijk}}.
\end{equation}
where  
$C_{ijk}=[(i-\sfrac{1}{2})h,(i+\sfrac{1}{2})h]\times[(j-\sfrac{1}{2})h,(j+\sfrac{1}{2})h]\times[(k-\sfrac{1}{2})h,(k+\sfrac{1}{2})h]$ and $\chi_A$ is the characteristic function of the set $A$. Then, we define 
piecewise constant interpolations in time, %
\begin{align}\label{eq:Q_h}
Q_{h,\Delta t}(t,x)&=\sum_{n=0}^{N_T-1}\,Q^n_{h,\Delta t}(x)\chi_{S_n}(t),\\
\label{eq:r_h}
r_{h,\Delta t}(t,x)&=\sum_{n=0}^{N_T-1}\,r^n_{h,\Delta t}(x)\chi_{S_n}(t),\\
\label{eq:P_h}
P_{h,\Delta t}(t,x)&=\sum_{n=0}^{N_T-1}\,P^n_{h,\Delta t}(x)\chi_{S_n}(t),
\end{align}
where $T=N_T\Delta t$ and $S_n=[n\Delta t, (n+1)\Delta t)$. 
We will show that a subsequence of these converges to a weak solution of~\eqref{eq:reformulation}:
\begin{theorem}
	\label{thm:convQr}
	The piecewise constant interpolations~\eqref{eq:Q_h}--\eqref{eq:P_h} computed using scheme~\eqref{3d_mc_nf_d_scheme} converge up to a subsequence to a weak solution of~\eqref{eq:reformulation} (as in Definition~\ref{def:weakreformulation}) as $h,\Delta t\to 0$.
\end{theorem} 
\begin{proof}
{\bf Step 1: Compactness.}

We apply the first order finite difference operator $D_t^+$ on $Q_{h,\Delta t}$ and $r_{h,\Delta t}$,
\begin{equation}
    D^+_tQ_{h,\Delta t}(t,x)=\sum_{n=0}^{N_T-1}\frac{Q^{n+1}_{h,\Delta t}(x)-Q^n_{h,\Delta t}(x)}{\Delta t}\chi_{S_n},
\quad\text{and }\quad
    D^+_tr_{h,\Delta t}(t,x)=\sum_{n=0}^{N_T-1}\frac{r^{n+1}_{h,\Delta t}(x)-r^n_{h,\Delta t}(x)}{\Delta t}\chi_{S_n}.\label{eq:DTr_h}
\end{equation}
From the energy stability of the scheme, Theorem~\ref{thm:energyestimate}, it follows that $\{\Grad_h Q_{h,\Delta t}\}\subset L^\infty(0,T;L^2(\dom))$
and $\{r_{h,\Delta t}\}\subset L^\infty(0,T;L^2(\dom))$
 uniformly in $\Delta t,h>0$. Corollary~\ref{cor:timeder} yields $\{D_t^+ Q_{h,\Delta t}\}\subset L^2([0,T]\times \dom)$ uniformly in $h,\Delta t>0$.
Moreover, from Lemma~\ref{lem:l2bound}, we get
	\begin{equation*}
\norm{Q_{h,\Delta t}(t)}_{L^2(\dom)}\leq T^\frac{1}{2}  \|D_t^+Q_{h,\Delta t}\|_{L^2([0,T]\times \Omega)}+\|Q_{h,\Delta t}(0)\|_{L^2(\dom)}<\infty,
\end{equation*}
and hence $\{Q_{h,\Delta t}\}\subset L^\infty([0,T];L^2(\dom))$.
Therefore, we can apply a discretized version of the Aubin-Lions lemma~\cite[Lemma A.1]{Trivisa2017}, to conclude that
there exists $Q\in L^2([0,T], H^1(\Omega))$ and a subsequence $\{Q_{h_m, \Delta t_m}\}_{m}$ such that $Q_{h_m, \Delta t_m}\to Q$ in $L^2([0,T]\times\Omega)$ as $m\to \infty$. Due to the uniform bounds, we also obtain $\nabla_{h_m}Q_{h_m, \Delta t_m}\rightharpoonup \nabla Q$ in $L^2([0,T]\times \Omega)$  
and we can extract a weakly convergent subsequence of $\{D^+_tQ_{h_m, \Delta t_m}\}_{m}$ and $\{r_{h_m, \Delta t_m}\}_{m}$, for simplicity still indexed by $m$.  In summary, we have the following:
\begin{equation}
    \begin{aligned}\label{eq:Qweakconvergence}
    &Q_{h_m, \Delta t_m}\to Q,\qquad\text{in } L^2([0,T]\times\Omega),\quad &D^+_tQ_{h_m, \Delta t_m}\rightharpoonup Q_t\qquad\text{in } L^2([0,T]\times\Omega),\\
    &\nabla_hQ_{h_m, \Delta t_m}\rightharpoonup \nabla Q,\qquad\text{in } L^2([0,T]\times\Omega),
  \end{aligned}
\end{equation}
and 
\begin{equation}\label{eq:rConvergence}
    r_{h_m, \Delta t_m}\overset{\ast}{\rightharpoonup} g,\quad \text{ in }\, L^\infty([0,T];L^2(\dom)).
\end{equation}
Since we have shown that $P(Q)$ is Lipschitz continuous with respect to $Q$ in Theorem~\ref{thm:PLipschitz}, we obtain from the strong convergence of $\{Q_{h_m,\Delta t_m}\}_m$ that 
\begin{equation}\label{eq:PQconvergence}
    P(\,Q_{h_m, \Delta t_m}\,)\to P(Q),\qquad\text{in } L^2([0,T]\times\Omega).
\end{equation}
{\bf Step 2: Passing to the limit $m\to\infty$.}

Next, we show that the sequences $\{Q_{h_m,\Delta t_m}\}_m$, $\{r_{h_m,\Delta t_m}\}_m$ converge to a weak solution of~\eqref{eq:reformulation}, that is, that the limit $(Q,r)$ is a weak solution in the sense of Definition~\ref{def:weakreformulation}.
We start with the equation for the variable $r$. 
From the numerical scheme~\eqref{3d_mc_nf_d_scheme}, it follows that
\begin{equation}\label{eq:Dtrhformula}
D^+_tr_{h, \Delta t}=P_{h,\Delta t}(t,x):D^+_tQ_{h, \Delta t}.
\end{equation} 
For any smooth test function $\phi$ with compact support in $[0,T]\times\Omega$, we have $P(Q)\,\phi\in L^2([0,T]\times\Omega)$. Therefore, using the weak convergence $D^+_tQ_{h_m, \Delta t_m}\rightharpoonup Q_t$ in $L^2([0,T]\times\Omega)$ we find that,
\begin{equation}\label{eq:PQDtconvergence}
\int_0^T\!\!\int_\Omega P(Q):D^+_tQ_{h_m, \Delta t_m}\phi \, dx dt\stackrel{m\to \infty}{\longrightarrow}\int_0^T\!\!\int_\Omega P(Q):Q_t \phi\, dx dt.
\end{equation}
Moreover, by the strong convergence of $P_{h_m,\Delta t_m}$ in $L^2([0,T]\times \dom)$, we have
\begin{align*}
&\quad\left|\int_0^T\int_\Omega (\,P_{h_m,\Delta t_m}-P(Q)\,):D^+_tQ_{h_m, \Delta t_m}\phi\,  dx dt\right|\\
&\leq\|\phi\|_{L^\infty
	(\Omega\times[0,T])}\,\|P_{h_m,\Delta t_m}-P(Q)\|_{L^2([0,T]\times\Omega)}\,\|D^+_tQ_{h_m,\Delta t_m}\|_{L^2([0,T]\times\Omega)}\stackrel{m\to \infty}{\longrightarrow} 0.
\end{align*}
Therefore, we can multiply $\phi$ on both sides of (\ref{eq:Dtrhformula}), integrate over both space $\Omega$ and  time interval $[0,T]$ and apply~\eqref{eq:PQDtconvergence} to obtain,
\begin{equation*}
\begin{aligned} \int_0^T\int_\Omega D^+_tr_{h_m, \Delta t_m}\phi\, dx dt&= \int_0^T\int_\Omega P_{h_m,\Delta t_m}:D^+_tQ_{h_m, \Delta t_m}\phi\, dx dt\\
&=\int_0^T\int_\Omega (\,P_{h_m,\Delta t_m}-P(Q)\,):D^+_tQ_{h_m, \Delta t_m}\phi\, dx dt\\
&\quad +\int_0^T\int_\Omega P(Q):D^+_tQ_{h_m, \Delta t_m}\phi\, dx dt\\
&\stackrel{m\to\infty}{\longrightarrow}\int_0^T\int_\Omega P(Q):Q_t\,\phi\, dx dt.
\end{aligned}
\end{equation*}
For the left hand side of~\eqref{eq:Dtrhformula}, we combine the definition of the piecewise constant functions, \eqref{eq:discrete_solution} and   \eqref{eq:DTr_h}, rename the integration variables so that the difference operator acts on the smooth test function:
\begin{align}\label{eq:lhs}
\text{LHS}&=\sum_{n=0}^{N_T-1}\int_{S_n}\int_\Omega\,\frac{r^{n+1}_{h_m,\Delta t_m}(x)-r^n_{h_m,\Delta t_m}(x)}{\Delta t_m}\,\phi(t,x)\, dx dt\notag\\
&=\sum_{n=1}^{N_T-1}\int_{S_n}\int_\Omega\,r^n_{h_m,\Delta t_m}\,\frac{\phi(x,t-\Delta t_m)-\phi(t,x)}{\Delta t_m}\,dx dt\notag\\
&\quad+\frac{1}{\Delta t_m}\,\int_{S_{N_T-1}}\int_\Omega r^{N_T}_{h_m,\Delta t_m}\,\phi(t,x)\,dx dt-\frac{1}{\Delta t_m}\,\int_{S_0}\int_\Omega r^{0}_{h_m,\Delta t_m}\,\phi(t,x)\,dx dt\notag\\
&=\int_0^T\int_\Omega r_{h_m,\Delta t_m}\,\frac{\phi(x,t-\Delta t_m)-\phi(t,x)}{\Delta t_m}\,dx dt\notag\\
&\quad+\frac{1}{\Delta t_m}\,\int_{S_{N_T-1}}\int_\Omega r^{N_T}_{h_m,\Delta t_m}\,\phi(t,x)\,dx dt-\frac{1}{\Delta t_m}\,\int_{S_0}\int_\Omega r^{0}_{h_m,\Delta t_m}\,\phi(t,x)\, dx dt.
\end{align}
When $\phi$ has compact support in $[0,T)\times \dom$, the second term on the right hand side vanishes, and we can use the  weak* convergence of $\{r_{h,\Delta t}\}$, \eqref{eq:rConvergence}, to pass the limit $h,\Delta t \to 0$,
\begin{align*}
\text{LHS}
\stackrel{m\to\infty}{\longrightarrow} -\int_0^T\int_\Omega\, g\phi_t dx dt-\int_\dom r^0(x) \phi(0,x) dx.
\end{align*}
Using~\cite[Lemma 1.1, p. 250]{Temam_NS}, this implies that $r$ is weakly continuous in time on $L^1(\dom)$, since $P(Q):Q_t\in L^2([0,T];L^1(\dom))$ by the Lipschitz continuity of $P$. Lemma~\ref{lem:weakcont} then implies that also $\int r_{h,\Delta t}(t,x)\phi(t,x) dx\to \int g(t,x)\phi(t,x) dx$ for every $t\in [0,T]$ up to a subsequence as $h,\Delta t\to 0$, and hence we can pass to the limit in the left hand side~\eqref{eq:lhs} when $\phi$ is compactly supported in $[0,T]\times\dom$.
Thus the limit $g$ satisfies~\eqref{eq:weakr}.

Next we show that the limit $Q$ satisfies~\eqref{eq:weakformQr}. We take the inner product of the first equation in~\eqref{3d_mc_nf_d_scheme} with a smooth matrix-valued function $\varphi = (\varphi_{\alpha\beta})_{\alpha,\beta=1}^d:[0,T]\times \Omega\to \R^{d\times d}$ integrated over $S_n\times C_{ijk}$, i.e., $\iint_{S_n\times C_{ijk}}\varphi\, dxdt$ and then sum over $n$ and $i,j,k$. We obtain
\begin{equation*}
\begin{aligned}
&\sum_{n=0}^{N_T-1} \sum_{i,j,k=1}^{N}\int_{S_n}\,\int_{C_{ijk}}\frac{Q_{ijk}^{n+1}-Q_{ijk}^n}{\Delta t}\,:\varphi\, dxdt\\
&=\sum_{n=0}^{N_T-1} \sum_{i,j,k=1}^{N}\int_{S_n}\,\int_{C_{ijk}}M\,(L_1\Delta_hQ_{ijk}^{n+\frac{1}{2}}-r^{n+\frac{1}{2}}_{ijk}\overline{P}_{ijk}^{n+\frac{1}{2}}+\frac{L_2+L_3}{2}\alpha_{ijk}^{\hlfstp})\,:\varphi\, dx dt.
\end{aligned}
\end{equation*}
We rewrite this in terms of the piecewise constant functions~\eqref{eq:discrete_solution}:
\begin{equation}\label{eq:discreteweakformulation_2}
\begin{aligned}
&\sum_{n=0}^{N_T-1} \int_{S_n}\,\int_{\Omega}D_t^+Q_{h_m,\Delta t_m}^{n}\,:\varphi\, dxdt\\
&=\sum_{n=0}^{N_T-1}  \int_{S_n}\,\int_{\Omega}M\,(L_1\Delta_hQ_{h_m,\Delta t_m}^{n+\frac{1}{2}}-r^{n+\frac{1}{2}}_{h_m,\Delta t_m}\overline{P}_{h_m,\Delta t_m}^{n+\frac{1}{2}}+\frac{L_2+L_3}{2}\alpha_h(Q_{h_m,\Delta t_m})^{\hlfstp})\,:\varphi\, dx dt.
\end{aligned}
\end{equation}
(Here $\alpha_h(Q_{h,\Delta t})^{\hlfstp}=\dfrac{1}{2}(\alpha_h(Q_{h,\Delta t}^n)+\alpha_h(Q_{h,\Delta t}^{n+1}))$.)
Since $\{D_t^+Q_{h,\Delta t}\}$ is weakly convergent in $L^2$, c.f.~\eqref{eq:Qweakconvergence}, we can pass to the limit $m\to\infty$ in the left hand side and obtain
\begin{equation*}
\sum_{n=0}^{N_T-1} \int_{S_n}\,\int_{\Omega}D_t^+Q_{h_m,\Delta t_m}^{n}\,:\varphi\, dxdt\longrightarrow \int_0^T\int_{\dom} Q_t: \varphi\, dx dt.
\end{equation*}
Integrating by parts, we obtain the left hand side of~\eqref{eq:weakformQr}.
To deal with the right hand side of~\eqref{eq:discreteweakformulation_2}, we introduce the discrete forward and difference operators $D_k^+$ and $D^-_{k}$ for matrix functions $\varphi=(\varphi_{\alpha \beta})_{\alpha\beta}$, $1\leq \alpha,\beta\leq d$. Similar to \eqref{def:operator}, $D_k^+$ denotes the forward difference in the coordinate direction $k$. 
For example, for $x=(x_1,x_2,x_3)$ and $k=1$, we define
\begin{equation*}
\left(D_1^\pm\varphi(x)\right)_{\alpha\beta}=\pm\,\frac{\varphi_{\alpha\beta}(t,x_1\pm h,x_2,x_3)-\varphi_{\alpha\beta}(t,x_1,x_2,x_3)}{h}.
\end{equation*}
In addition, we introduce the discrete gradient and divergence operators for smooth $\varphi$:
\begin{align*}
&(\Grad_h^\pm \varphi)_{\alpha\beta}=\left((D^\pm_1 \varphi)_{\alpha\beta},\,(D^\pm_2 \varphi)_{\alpha\beta},\,(D^\pm_3 \varphi)_{\alpha\beta}\right)^\top,
&(\Div_h \varphi)_\beta= \sum_{\alpha=1}^d (D^c_i \varphi)_{\alpha\beta}
\end{align*}
where $\varphi_{\alpha\beta}$ is the $(\alpha,\beta)$-entry of the matrix $\varphi$. 
Renaming the integration variables such that the difference operators act on the test functions in the right hand side of~\eqref{eq:discreteweakformulation_2} and then using \eqref{eq:Qweakconvergence} and \eqref{eq:rConvergence}, the right hand side of \eqref{eq:discreteweakformulation_2} satisfies
\begin{align*}
\text{RHS}&=-ML_1\sum_{n=0}^{N_T-1}\int_{S_n}\int_{\dom}\nabla^-_{h_m} Q_{h_m,\Delta t_m}^{n+\frac{1}{2}}\cdot\nabla_{h_m}^- \varphi- M\sum_{n=0}^{N_T-1} \int_{S_n}\int_{\dom}r^{n+\frac{1}{2}}_{h_m,\Delta t_m}\overline{P}_{h_m,\Delta t_m}^{n+\frac{1}{2}}:\varphi \, dx dt\\
&\quad-M\frac{(L_2+L_3)}{2}\sum_{n=0}^{N_T-1} \int_{S_n}\int_{\dom}\sum_{\alpha,\beta,\gamma=1}^d \left(D_\gamma^c Q_{h_m,\Delta t_m}^n\right)_{\beta\gamma}(D_\alpha^c\varphi)_{\alpha\beta}\\
&\quad-M\frac{(L_2+L_3)}{2}\sum_{n=0}^{N_T-1} \int_{S_n}\int_{\dom}\sum_{\alpha,\beta,\gamma=1}^d  \left(D_{\gamma}^c Q_{h_m,\Delta t_m}^n \right)_{\alpha\gamma}(D_{\beta}^c \varphi)_{\alpha\beta}\\
&\quad+\frac{M(L_2+L_3)}{d}\sum_{n=0}^{N_T-1} \int_{S_n}\int_{\dom}\sum_{\alpha\beta,\gamma=1}^d  \left(D_{\alpha}^c Q_{h_m,\Delta t_m}^n\right)_{\gamma\alpha}(D_{\gamma}^c\varphi)_{\beta\beta}\\
&\stackrel{m\to\infty}{\longrightarrow}-ML_1\int_0^T\int_\Omega\sum_{\alpha,\beta=1}^d\Grad Q_{\alpha\beta}\cdot\Grad\varphi_{\alpha\beta}\,dx\,dt-M\lim\limits_{m\to \infty}\int_0^T \int_\Omega r_{h_m,\Delta t_m}\,P_{h_m,\Delta t_m}:\varphi\,dx\,dt\\
&\quad-M\,\frac{L_2+L_3}{2} \int_0^T\int_\Omega  \sum_{\alpha,\beta,\gamma=1}^d \left(\partial_{\gamma}Q_{\beta\gamma}\partial_\alpha \varphi_{\alpha\beta} + \partial_{\gamma}Q_{\alpha\gamma}\partial_\beta\varphi_{\alpha\beta} -\frac{2}{d}  \partial_{\alpha}Q_{\gamma\alpha} \partial_\gamma\varphi_{\beta\beta}\right)\,dxdt,
\end{align*}
where $\nabla_h^- Q_{h,\Delta t}^{n+\frac{1}{2}}\cdot\nabla^-_h\varphi=\sum_{\alpha,\beta=1}^d\,\nabla_h^- (Q_{h,\Delta t}^{n+\frac{1}{2}})_{\alpha \beta}\cdot(\nabla_h^-\varphi)_{\alpha \beta}.$ 
It remains to show

\begin{equation}\label{eq:rhlimit}
\lim\limits_{m\to\infty}\int_0^T \int_\Omega r_{h_m,\Delta,t_m}\,P_{h_m,\Delta t_m}:\varphi\,dx\,dt=\int_0^T \int_\Omega g\,P(Q):\varphi\,dx\,dt.
\end{equation}
To prove \eqref{eq:rhlimit}, we take the difference of the two terms, that is,
\begin{align*}
&\left|\int_0^T \int_\Omega r_{h_m,\Delta t_m}\,P_{h_m,\Delta t_m}:\varphi\,dx\,dt-\int_0^T \int_\Omega g\,P(Q):\varphi\,dx\,dt\right|\\
&=\left|\int_0^T \int_\Omega r_{h_m,\Delta t_m}\,(P_{h_m,\Delta t_m}-P(Q)):\varphi\,dx\,dt-\int_0^T \int_\Omega (g-r_{h_m,\Delta t_m})\,P(Q):\varphi\,dx\,dt\right|\\
&\leq\underbrace{\left|\int_0^T \int_\Omega r_{h_m,\Delta t_m}\,(P_{h_m,\Delta t_m}-P(Q)):\varphi\,dx\,dt\right|}_{\Romannum{1}}+\underbrace{\left|\int_0^T \int_\Omega (g-r_{h_m,\Delta t_m})\,P(Q):\varphi\,dx\,dt\right|}_{\Romannum{2}}.
\end{align*}
By Cauchy-Schwarz inequality, \eqref{eq:PQconvergence}, and the energy estimate, Theorem~\ref{thm:energyestimate},
\begin{equation*}
\Romannum{1}\leq\|\varphi\|_{L^\infty
	(\Omega\times[0,T])}\,\|P_{h_m,\Delta t_m}-P(Q)\|_{L^2([0,T]\times\Omega)}\,\|r_{h_m, \Delta t_m}\|_{L^2([0,T]\times\Omega)}\to0.
\end{equation*}
Note that $P(Q)\,\varphi\in L^2([0,T]\times\Omega)$ and $r_{h_m,\Delta t_m}\rightharpoonup g$ in $L^2$,
therefore $\Romannum{2}\to0$. This proves \eqref{eq:rhlimit}.
 Combining the estimates for the left and the right hand side, we see that $Q$ satisfies~\eqref{eq:weakformQr}.
 The trace-free condition and the symmetry are linear constraints and therefore conserved under the $L^2$-convergence of $Q_{h,\Delta t}$. The energy inequality is a direct result by passing the limits in Theorem ~\ref{thm:energyestimate} and using Fatou's lemma. Hence the limit $(Q,r)$ is a weak solution in the sense of Definition~\ref{def:weakreformulation}.
\end{proof}

\subsection{Equivalence of weak formulations ($r=r(Q)$)}
Now that we have established that the scheme converges to a weak solution of~\eqref{eq:reformulation}, it remains to show that such a weak solution is in fact a weak solution of~\eqref{eq:qtensorflow}. To do so, we show that the limit $g$ established above in~\eqref{eq:rConvergence} satisfies $g=r(Q)$ weakly, where $Q$ is the limit of $Q_{h_m,\Delta t_m}$ and 
$r(Q)$ is defined in~\eqref{eq:defr}. Plugging this into the weak formulation~\eqref{eq:weakform}, we see that $Q$ is in fact a weak solution in the sense of Definition~\ref{def:weakQ}. We thus need to prove the following lemma:
\begin{lem}\label{lem:requivalance}
	Assume that $(Q,g)$ is a weak solution in the sense of Definition~\ref{def:weakreformulation}. Then for any smooth $\psi$ with compact support in $(0,T)\times\Omega$ (compactly supported in both time and space),	we have
\begin{equation*}
\int_0^T\int_\Omega\, g\,\psi dx dt=\int_0^T\int_\Omega\,
\,r(Q)\,\psi \, dx dt,
\end{equation*}
where $r(Q)$ is defined in~\eqref{eq:defr}.
\end{lem}
\begin{proof}
Since $(Q,g)$ is a weak solution of~\eqref{eq:reformulation}, we have that
\begin{equation}\label{eq:weakgr}
-\int_0^T\int_{\dom}g\psi_t\, dx dt = \int_0^T\int_{\dom} P(Q):Q_t \,\psi \,dx dt
\end{equation}
for $\psi$ smooth and compactly supported in $(0,T)\times\dom$.
For the right hand side, if $Q$ is a smooth function, we can use chain rule and integration by parts to get
\begin{equation*}
    \int_0^T\int_\Omega P(Q):Q_t\,\psi=\int_0^T\int_\Omega\,r(Q)_t\,\psi=-\int_0^T\int_\Omega r(Q)\psi_t.
\end{equation*}
Since $Q\in L^2([0,T],H^1(\Omega))$ and $Q_t\in L^2([0,T]\times\dom)$, we can find a sequence of smooth function $\left\{Q_n\right\}_n$ with $Q_n\to Q$ in $L^2([0,T],H^1(\Omega))$ and $(Q_n)_t\to Q_t$ in $L^2([0,T]\times\Omega)$. 
We note that by mean value theorem,
\begin{equation*}
    r(Q)-r(Q_n)=P(\Tilde{Q}):(Q-Q_n),
\end{equation*}
for some $\Tilde{Q}=\lambda_1 Q+\lambda_2 Q_n$ where $\lambda_1,\lambda_2\in[0,1]$ and $\lambda_1+\lambda_2=1$. Noting that $P(Q) $ is Lipschitz continuous with respect to $Q$, so  $|P(\Tilde{Q})|_F\leq \tilde{L}|\Tilde{Q}|_F$ for some constant $\tilde{L}>0$. Therefore,
\begin{equation*}
   \left|r(Q)-r(Q_n)\right|_F=\left|P(\Tilde{Q}):(Q-Q_n)\right|_F\leq\tilde{L}\,(\left|Q\right|_F+\left|Q_n\right|_F)\,\left|Q-Q_n\right|_F.
\end{equation*}
Integrating it over time and space and we obtain
\begin{align*}
    \|r(Q)-r(Q_n)\|_{L^1([0,T]\times\Omega)}&\leq\tilde{L}\,\left(\|Q\|_{L^2([0,T]\times\Omega)}+\|Q_n\|_{L^2([0,T]\times\Omega)}\right)\,||Q-Q_n||_{L^2([0,T]\times\Omega)}\to0,
\end{align*}
since $Q$ and $Q_n$ are both bounded in $L^2([0,T]\times \Omega)$.
So if we use smooth functions to approximate $Q$, we obtain,
\begin{equation*}
  \left|\int_0^T\int_\Omega \,(\,r(Q)-r(Q_n)\,)\,\psi_t\right|\leq \|\psi_t\|_{L^\infty([0,T]\times\Omega)}\,\left\|r(Q)-r(Q_n)\right\|_{L^1([0,T]\times\Omega)}\to0.
\end{equation*}

On the other hand, using the Lipschitz continuity of $P(Q)$, we arrive at
\begin{align*}
    &\left|\int_0^T\int_\Omega\,(\,P(Q):Q_t\,\psi-P(Q_n):(Q_n)_t\,\psi\,)\right|\\
    &\leq\left|\int_0^T\int_\Omega\,(\,P(Q):Q_t-P(Q):(Q_n)_t\,)\,\psi\right|+\left|\int_0^T\int_\Omega\,(\,P(Q):(Q_n)_t-P(Q_n):(Q_n)_t\,)\,\psi\right|\\
    &\leq \|P(Q)\|_{L^2}\,\|\psi\|_{L^\infty}\,\|Q_t-(Q_n)_t\|_{L^2}+\|P(Q)-P(Q_n)\|_{L^2}\,\|\psi\|_{L^\infty}\,\|(Q_n)_t\|_{L^2} \stackrel{n\to\infty}{\longrightarrow}0.
\end{align*}
Therefore, we have for any $Q\in L^2([0,T],H^1(\Omega))$ with $Q_t\in L^2([0,T]\times\dom)$
\begin{equation*}
    \int_0^T\!\!\int_\Omega P(Q):Q_t\,\psi=\lim\limits_{n\to\infty}\int_0^T\!\!\int_\Omega P(Q_n):(Q_n)_t\,\psi=-\lim\limits_{n\to\infty}\int_0^T\!\!\int_\Omega r(Q_n)\psi_t=-\int_0^T\!\!\int_\Omega r(Q)\psi_t.
\end{equation*}
We use this in~\eqref{eq:weakgr} to obtain
\begin{equation*}
    \int_0^T\int_\Omega\, g\psi_t=
    -\int_0^T\int_\Omega P(Q):Q_t\,\psi=\int_0^T\int_\Omega r(Q)\psi_t.
\end{equation*}
From~\cite[Lemma 1.1, p. 250]{Temam_NS}, we obtain that $g$ as well as $r(Q)$ are absolutely continuous and satisfy for every test function $\psi\in L^\infty(\dom)$ and almost every $t\in [0,T]$
\begin{equation*}
    \int_\Omega g(t,x)\psi(x)\,dx=\int_\Omega r(Q(t,x))\psi(x)\,dx+\int_\Omega f(x)\psi(x)\,dx
\end{equation*} 
for some $f\in L^2(\Omega)$. 
However, since $g$ satisfies \eqref{eq:weakr}, by letting $T$ be $0$ in \eqref{eq:weakr}, we find 
\begin{equation*}
    \int_\Omega g(0,x)\psi(x)\,dx=\int_\Omega r^0(x)\psi(x)\,dx.
\end{equation*}
and so $f=0$ in $L^2(\Omega)$. This proves the lemma.
 
\end{proof}
This lemma shows that
\begin{equation*}
    \int_0^T \int_\Omega g\,P(Q):\varphi\,dx\,dt=\int_0^T \int_\Omega r(Q)\,P(Q):\varphi\,dx\,dt,
\end{equation*}
for any smooth and compactly supported $\varphi:[0,T]\times \dom\to \R^{d\times d}$. Plugging this into~\eqref{eq:weakformQr}, we see that the identity becomes~\eqref{eq:weakform} and hence any weak solution in the sense of Definition~\ref{def:weakreformulation} is in fact a weak solution in the sense of Definition~\ref{def:weakQ}. Hence we have shown:
\begin{theorem}
	\label{thm:convtoqtensorflow}
	Approximations computed by the numerical scheme~\eqref{3d_mc_nf_d_scheme} converge as $\Delta t,h\to 0$, up to a subsequence, to weak solutions of~\eqref{eq:qtensorflow} as in Definition~\ref{def:weakQ}.
\end{theorem}

\section{Numerical results in 2D}\label{sec:simulation}
We shall now present some numerical experiments in 2D. In this case, the term $\alpha(Q)$ in~\eqref{eq:QZ} simplifies to $\alpha(Q)= \Delta Q$. We therefore denote $L:=L_1+\dfrac{1}{2}(L_2+L_3)$. We will use the parameters
\begin{equation}\label{eq:parameters}
a = -0.3 \quad b = -4 \quad c = 4 \quad A_0 = 500 \quad M =1,
\end{equation}
unless specified otherwise. The scheme has been implemented in MATLAB and the code used to run the following numerical examples can be found at~\url{github.com/VarunMG/Liquid-Crystal-Energy-Stable}.
 \subsection{Numerical Example 1: Convergence test}
 First we check whether the formal second order of accuracy of the scheme manifests in practice when simulating a numerical example with smooth solution. We consider the domain $\dom=[0,2]^2$, $L=0.001$ and the initial condition
 \begin{equation}
     \label{eq:smoothinit1}
     Q_0 = \mathbf{n}_0 \mathbf{n}_0^\top - \frac{|\mathbf{n}_0|^2}{2} I_2,
 \end{equation}
 where
 \begin{equation}
     \label{eq:smoothinit2}
     \mathbf{n_0}(x,y) =\begin{pmatrix}  x(2-x)y(2-y)\\
\sin(\pi x)\sin(0.5\pi y).
     \end{pmatrix}
 \end{equation}
 \subsubsection{Refinement in space} We compute up to time $T=0.4$ using 400 time steps and we will use a reference solution $(Q^{\text{ref}},r^{\text{ref}})$ to show the spatial accuracy of our scheme. The reference solution is computed with $400$ grid points in each spatial direction and $4000$ time steps. All the errors are measured in $L^2$-norm
 \begin{equation*}
     \mathcal{E}^Q_{\alpha\beta} = \norm{Q_{\alpha\beta}^{\text{ref}}(T,\cdot)-(Q_h)_{\alpha\beta}(T,\cdot)}_{L^2(\dom)}, \quad  \mathcal{E}^r = \norm{r^{\text{ref}}(T,\cdot)-r_h(T,\cdot)}_{L^2(\dom)} \end{equation*}
 where $\alpha,\beta\in \{1,2\}$.
  We compute the numerical solutions with $n=10,20,40,80,$ grid points in each spatial direction. The $L^2$-errors and convergence rates for $Q_{11}, Q_{12}$ and $r$ are reported in Table~\ref{tab:convspace}. (Note that due to the symmetry and the trace-free property, $Q_{11}=-Q_{22}$ and $Q_{12}=Q_{21}$.) We note that for the components of $Q$ the expected second order convergence rate is almost achieved whereas the convergence rate for the variable $r$ is lower. We suspect that more mesh refinement may be needed to see the optimal order for the variable $r$.
  \begin{table}[H]
      \centering
      \begin{tabular}{|c|c|c|c|c|c|c|}
      \hline
           $h$ &        error for $Q_{11}$  & order for $Q_{11}$ &   error for $Q_{12}$   & order for $Q_{12}$ & error for $r$ &  order for $r$\\
           \hline 
   $0.2$  & $1.3509\times 10^{-2}$ &  NaN &      $2.3646\times 10^{-2}$ & NaN & $6.7561\times 10^{-3}$ & NaN\\
   $0.1$  & $3.7509\times 10^{-3}$ & $1.8486$  &    $6.4006\times 10^{-3}$ & $ 1.8854$ & $1.2878\times 10^{-3}$ & $2.3912$ \\
   $0.05$  & $9.9049\times 10^{-4}$ &  $1.9210$   &      $1.6690\times 10^{-3}$ & $1.9392$ & $3.8885\times10^{-4}$ & $1.7277$\\
   $0.025$  & $2.6162\times 10^{-4}$ &  $1.9206$ &        $4.4341\times 10^{-4}$ & $1.9123$ & $1.5189\times 10^{-4}$&     $1.3562$ \\
   \hline
      \end{tabular}
      \caption{Errors and rates for spatial refinement in example~\eqref{eq:smoothinit1},~\eqref{eq:smoothinit2}.}
      \label{tab:convspace}
  \end{table}

Figure~\ref{fig:energyexample1} shows the decay of the discrete energy for $n=80$ and $N_T=400$ up to time $T=0.4$. As predicted by the theory, the energy decays monotonically.

 \begin{figure}
     \centering
     \includegraphics[width=4in]{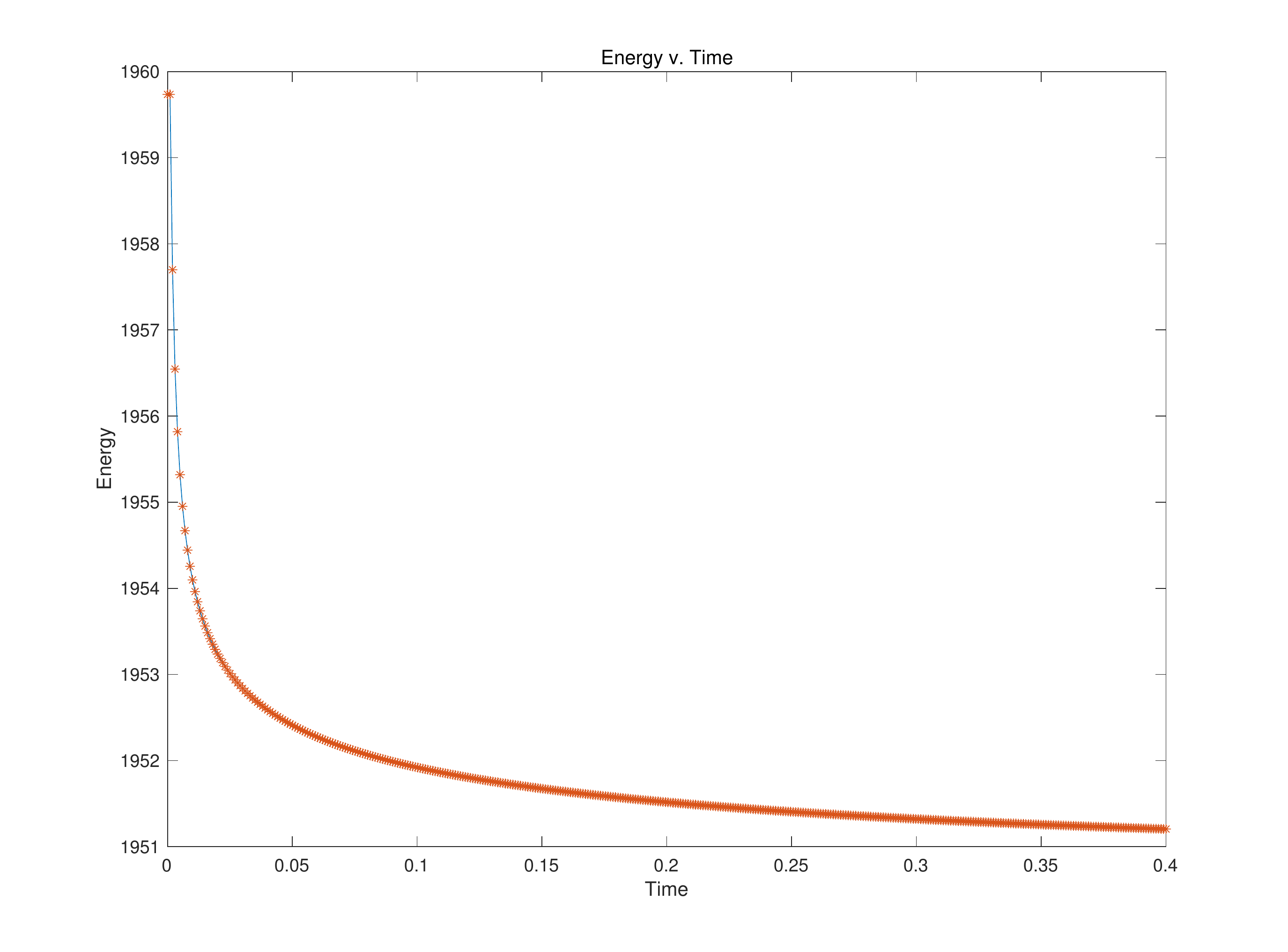}
     \caption{Energy decay when $T=0.4$ with $400$ time steps and $80$ grid points in each spatial direction.}

     \label{fig:energyexample1}
 \end{figure}

 \subsubsection{Refinement in time}
 We use the same setting (initial value and parameters) as for the spatial accuracy test and compute up to time $T=0.4$ with 100 grid points in each spatial direction. Similar as above, we will compare the approximations with different time step sizes with a reference solution which is computed with the same number of spatial points and 8000 time steps.
 The errors and convergence orders for the approximations with 40, 80, 160, 320, 640 time steps are shown in Table~\ref{tab:convtime}. We observe second order accuracy as expected.
  \begin{table}[H]
      \centering
      \begin{tabular}{|c|c|c|c|c|c|c|}
      \hline
           $\Delta t$ &        error for $Q_{11}$    &      order for $Q_{11}$ & error for $Q_{12}$ & order for $Q_{12}$ & error for $r$ & order for $r$\\
           \hline 
   $0.01$  & $7.87395\times 10^{-4}$ &  NaN &   $1.49178\times 10^{-3}$   & NaN & $9.15588\times 10^{-4}$& NaN\\
   $5\times 10^{-3}$  & $1.94110\times 10^{-4}$ &  $2.02022$    &    $3.67711\times 10^{-4}$ & 2.02039 & $2.19902\times 10^{-4}$ & 2.05784\\
   $2.5\times 10^{-3}$  & $4.81199\times 10^{-5}$ &  $2.01217$    &    $9.11505\times 10^{-5}$ & 2.01225 & $5.38631\times 10^{-5}$ & 2.02949 \\
   $1.25\times10^{-3}$  & $1.19280\times 10^{-5}$ &  2.01223 &        $2.25937\times 10^{-5}$ & 2.01233 & $1.32752\times 10^{-5}$ & 2.02056\\ 
   $6.25\times10^{-4}$  & $2.91895\times 10^{-6}$ &  2.03083 &        $5.52893\times 10^{-6}$ & 2.03085 & $3.23961\times 10^{-6}$ & 2.03485\\
   $3.125\times10^{-4}$  & $6.71610\times 10^{-7}$ &  2.11975 &        $1.27212\times 10^{-6}$ & 2.11976 & $7.44387\times 10^{-7}$ & 2.12170\\
   \hline
      \end{tabular}
      \caption{Errors and rates for time refinement in example~\eqref{eq:smoothinit1},~\eqref{eq:smoothinit2}}
      \label{tab:convtime}
  \end{table}
 
 \subsection{Numerical Example 2: Defects in Liquid Crystals}
 
We consider the domain $\Omega=[0,2] \times [0,2]$ and $L=0.001$. In this example, we will study the dynamics of defects in liquid crystals.
 For the initial condition, we take
 \begin{equation}\label{eq:q0}
Q_0 = \mathbf{n}_0 \mathbf{n}_0^\top - \frac{|\mathbf{n}_0|^2}{2} I_2,
\end{equation}
where
 \begin{equation}
     \label{eq:smoothinit3}
     \mathbf{n_0}(x,y) =\begin{pmatrix}  \log(x^2+1)\,(x-2)^2\,\sin(\frac{\pi y}{2})\,(e^{1.5}-e^x)\\
(y-2)\,(y-3)\,\sin(\frac{\pi y}{10})\,\sin(\frac{\pi x}{2})\,(0.7-y)
     \end{pmatrix}.
 \end{equation}
We use 40 grid points in space in each dimension and 4000 time steps up to $T=4$. 
 As we can see from Figure~\ref{fig:defect}, initially, there is only one defect, which is located at $(1.5, 0.7)$. This configuration is not stable and generally splits into two different defects. They move away from each other and towards the boundary. 
 Figure~\ref{fig:defect_eig} depicts the largest eigenvalue of matrix $Q$ at different times. We observe that as the two defects move, the largest eigenvalue decays in a neighborhood of the defects rapidly to $0$. The eigenvalue is generally decreasing and tends to $0$ everywhere as time evolves. This behaviour is a consequence of the boundary condition and the energy dissipation property. 
 
 \begin{figure} [h]
	\begin{tabular}{ccc}
		\subfloat[$t=0$]{\includegraphics[width = 1.8in]{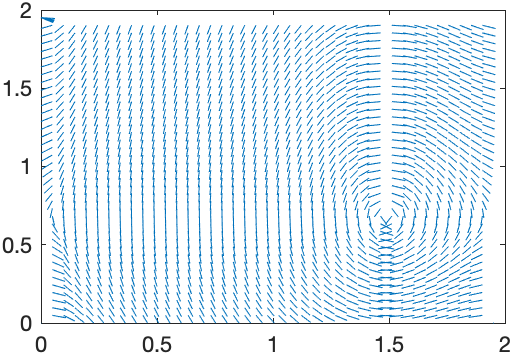}} &
		\subfloat[$t=0.5$]{\includegraphics[width = 1.8in]{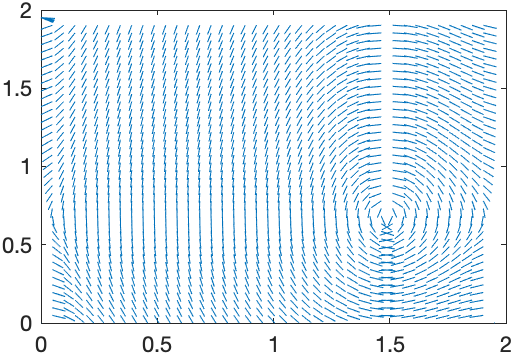}}
		\subfloat[$t=1$]{\includegraphics[width = 1.8in]{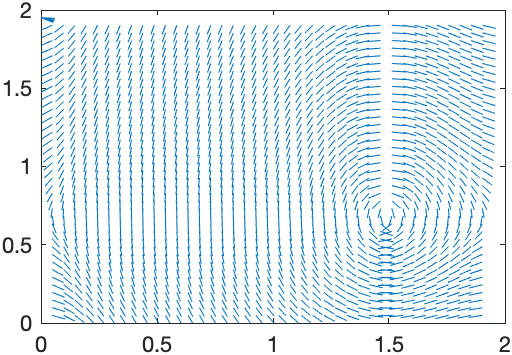}} \\
		\subfloat[$t=1.5$]{\includegraphics[width = 1.8in]{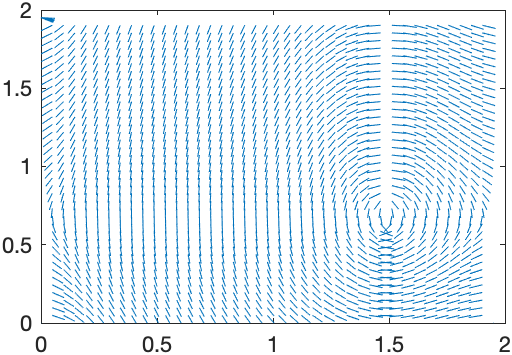}}&
		\subfloat[$t=2$]{\includegraphics[width = 1.8in]{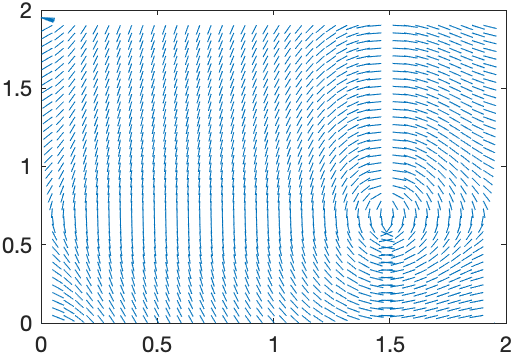}} 
		\subfloat[$t=2.5$]{\includegraphics[width = 1.8in]{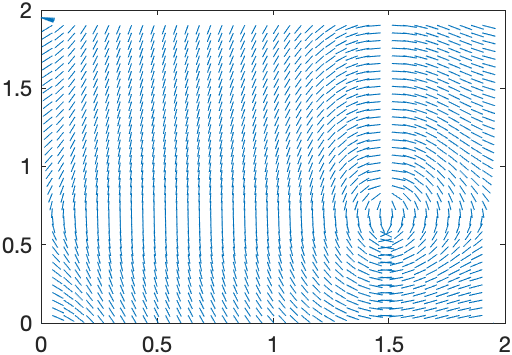}}\\
		
		\subfloat[$t=3$]{\includegraphics[width = 1.8in]{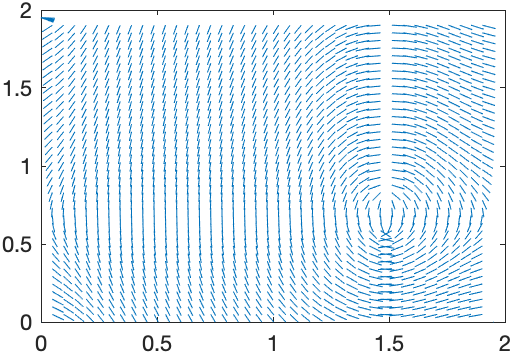}}&
		
		\subfloat[$t=3.5$]{\includegraphics[width = 1.8in]{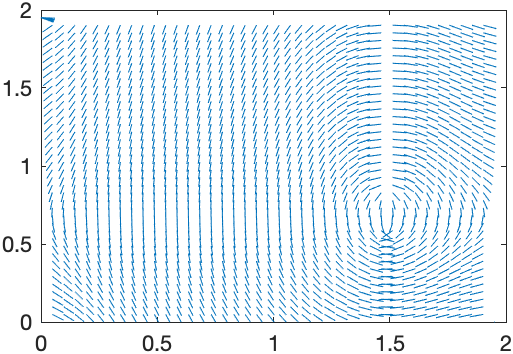}}
		
		\subfloat[$t=4$]{\includegraphics[width = 1.8in]{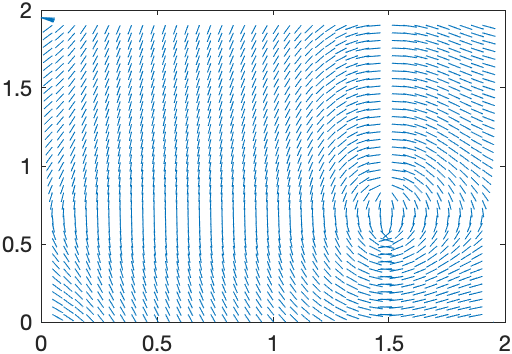}}
			\end{tabular}
	\caption{Simulation for initial data~\eqref{eq:q0}, \eqref{eq:smoothinit3}.}
	\label{fig:defect}
\end{figure}

 \begin{figure} [h]
	\begin{tabular}{ccc}
		\subfloat[$t=0$]{\includegraphics[width = 1.8in]{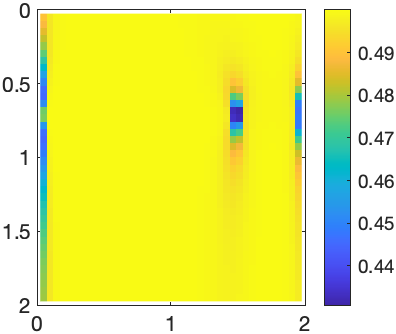}} &
		\subfloat[$t=0.5$]{\includegraphics[width = 1.8in]{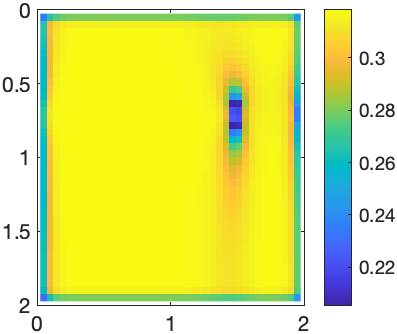}}
		\subfloat[$t=1$]{\includegraphics[width = 1.8in]{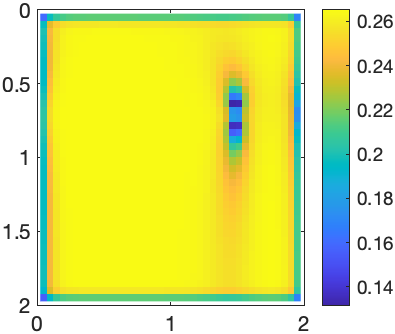}} \\
		\subfloat[$t=1.5$]{\includegraphics[width = 1.8in]{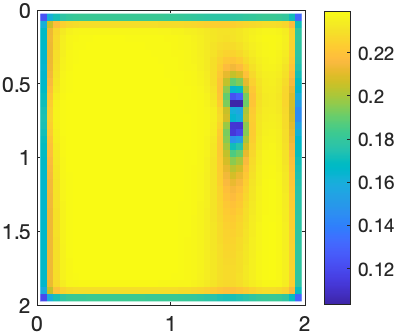}}&
		\subfloat[$t=2$]{\includegraphics[width = 1.8in]{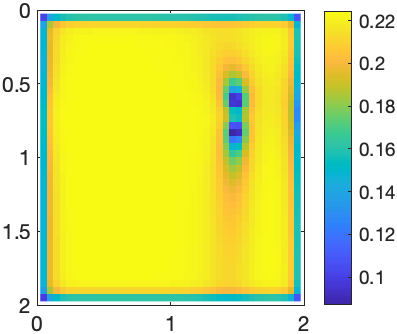}} 
		\subfloat[$t=2.5$]{\includegraphics[width = 1.8in]{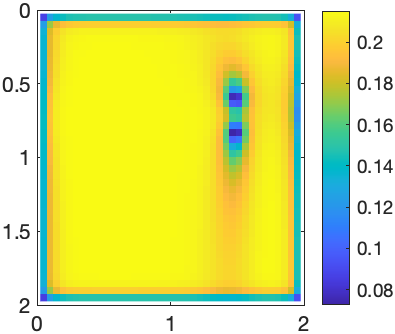}}\\
		
		\subfloat[$t=3$]{\includegraphics[width = 1.8in]{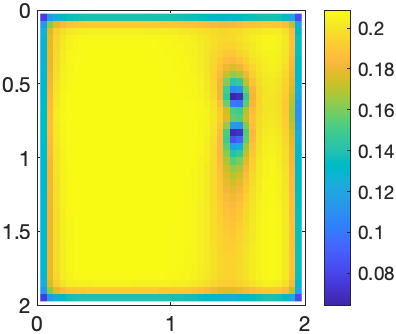}}&
		
		\subfloat[$t=3.5$]{\includegraphics[width = 1.8in]{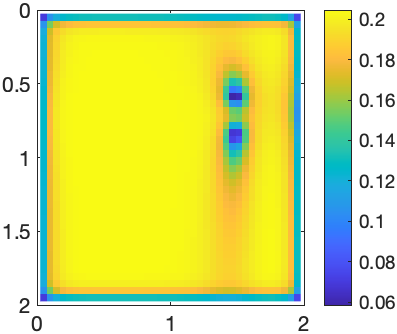}}
		
		\subfloat[$t=4$]{\includegraphics[width = 1.8in]{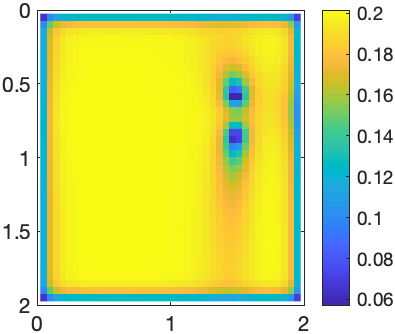}}
			\end{tabular}
	\caption{The largest eigenvalue of $Q$ in the simulation for initial data~\eqref{eq:q0}, \eqref{eq:smoothinit3}.}
	\label{fig:defect_eig}
\end{figure}

\subsection{Numerical Example 3: `Disappearing hole'}
We consider $\Omega=[0,1] \times [0,1]$ and use the parameters $a = -0.2, b=1, c=1, L=0.0025$. As an initial condition, we use~\eqref{eq:q0} with
\begin{equation}
\label{eq:init1}
\widetilde{\mathbf{n}}_0(x,y) =\begin{pmatrix}
x(1-x)y(1-y)\\ \sin(2\pi x)\sin(2\pi y);
\end{pmatrix},\quad \mathbf{n}_0= \frac{\widetilde{\mathbf{n}}_0}{|\widetilde{\mathbf{n}}_0|},
\end{equation}
and $50$ grid points in space in each dimension and $100$ time steps. The simulation is displayed in Figure~\ref{fig:hole}. We observe that the initial misalignment disappears first along the axes and then propagates in a shrinking circle towards the center of the domain and eventually disappears. This behavior was stable with respect to mesh refinement. The discrete energy~\eqref{eq:energydisc} decays at first rapidly and then approaches a constant state corresponding to the alignment of the director field along the $y$-axis as seen in Figure~\ref{disappearingholeenergey}.
\begin{figure} [h]
	\begin{tabular}{ccc}
		\subfloat[$t=0$]{\includegraphics[width = 2in]{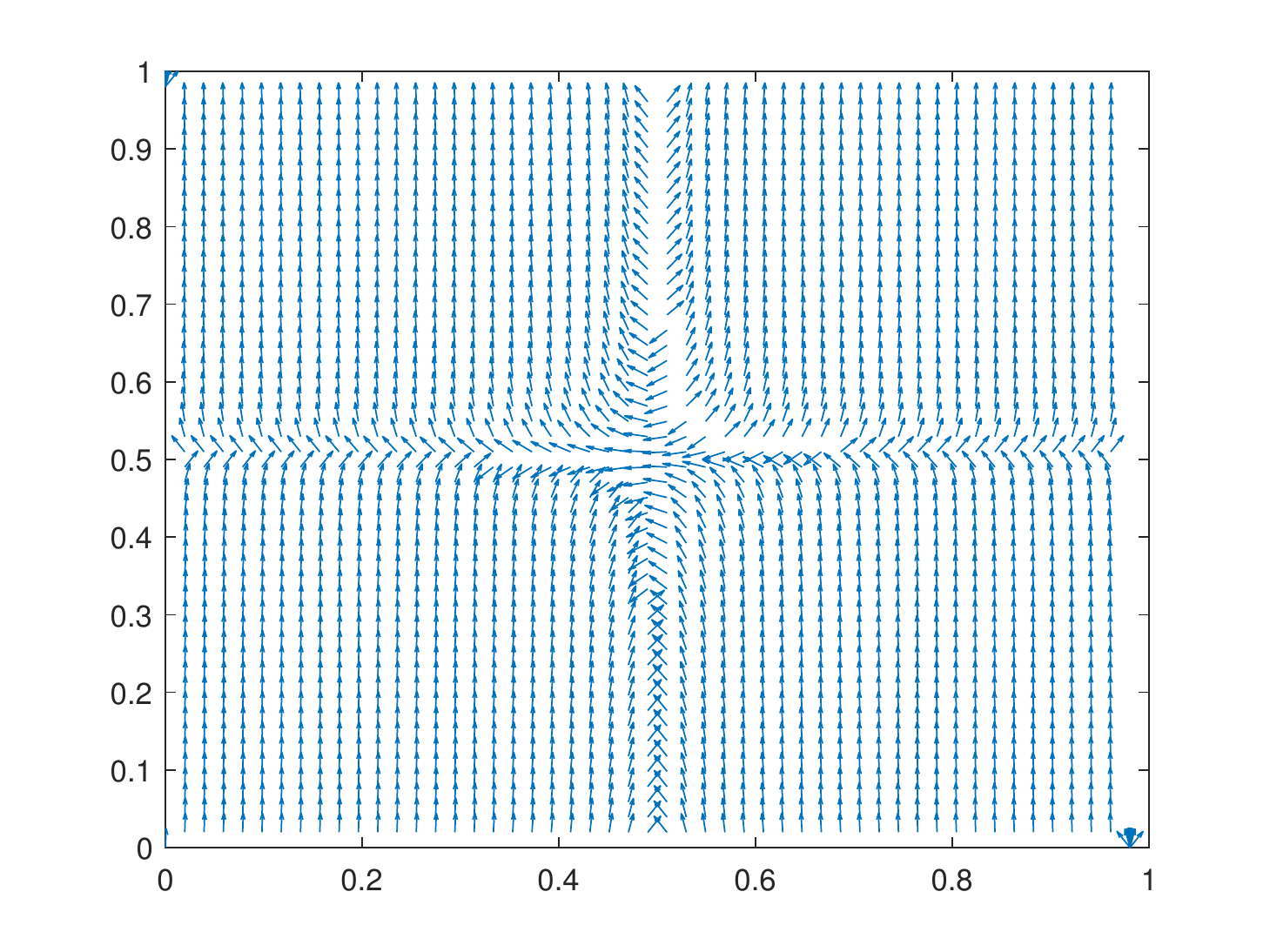}} &
		\subfloat[$t=0.2$]{\includegraphics[width = 2in]{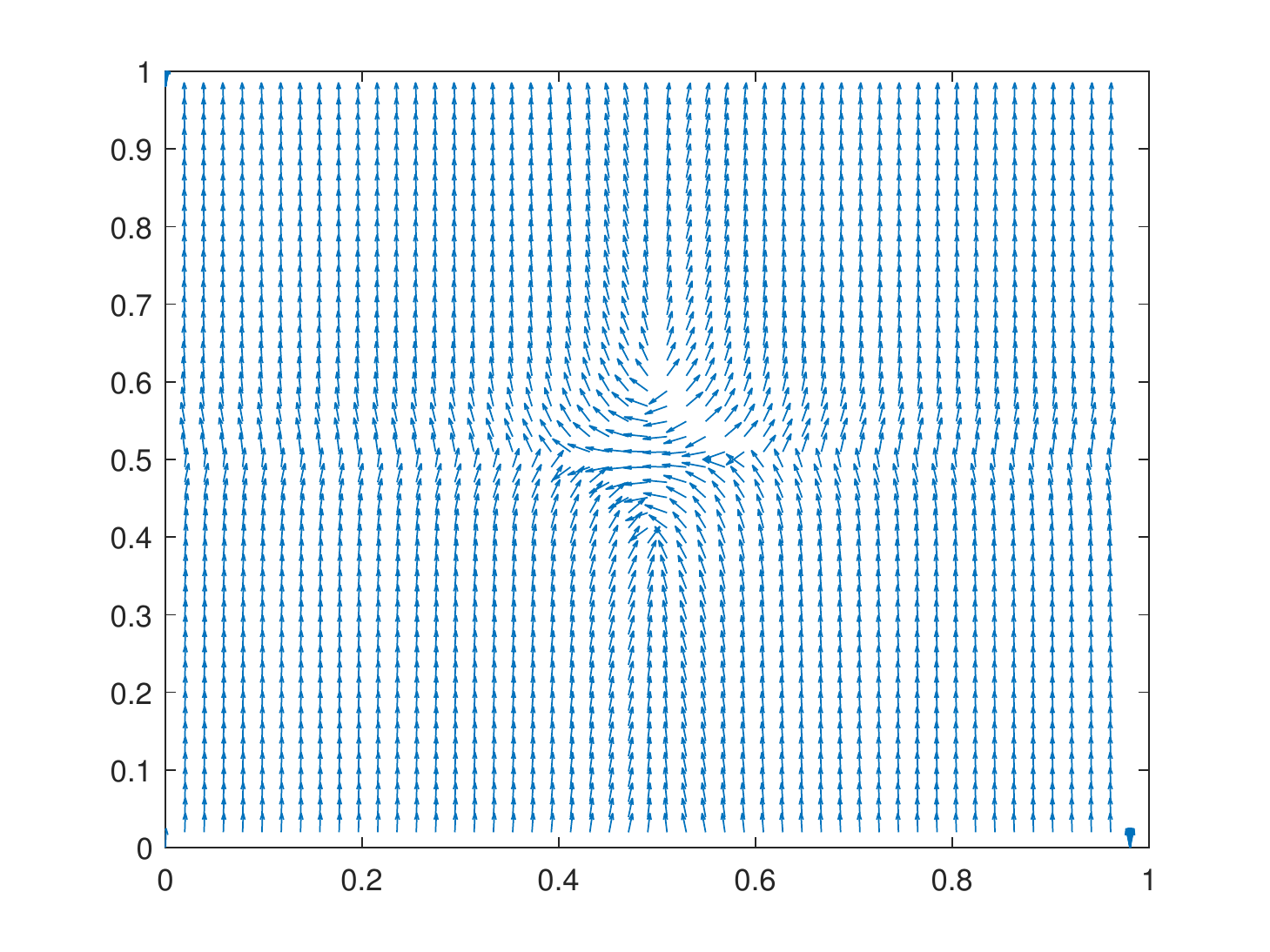}} &
		\subfloat[$t=0.4$]{\includegraphics[width = 2in]{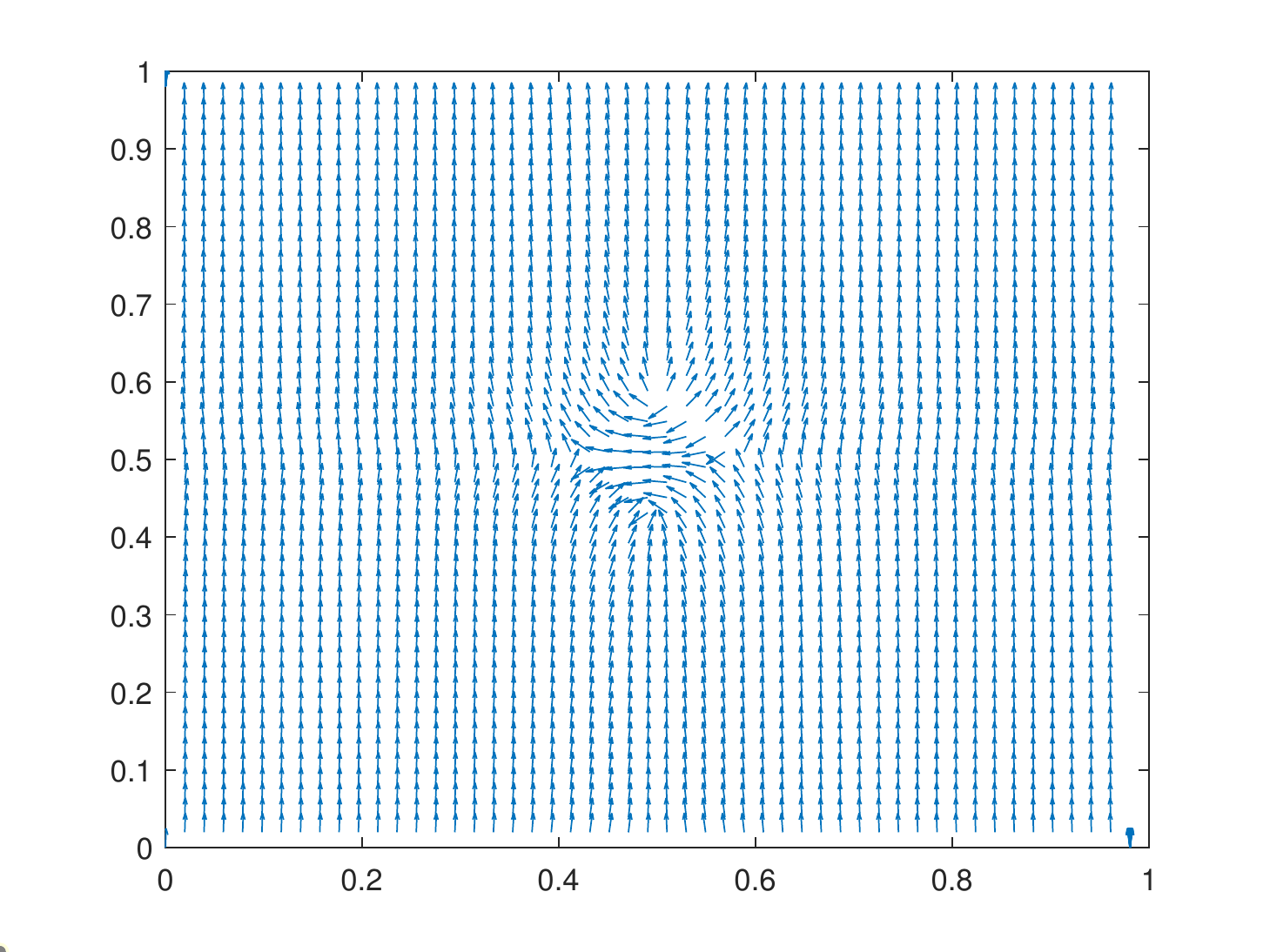}} \\
		\subfloat[$t=0.6$]{\includegraphics[width = 2in]{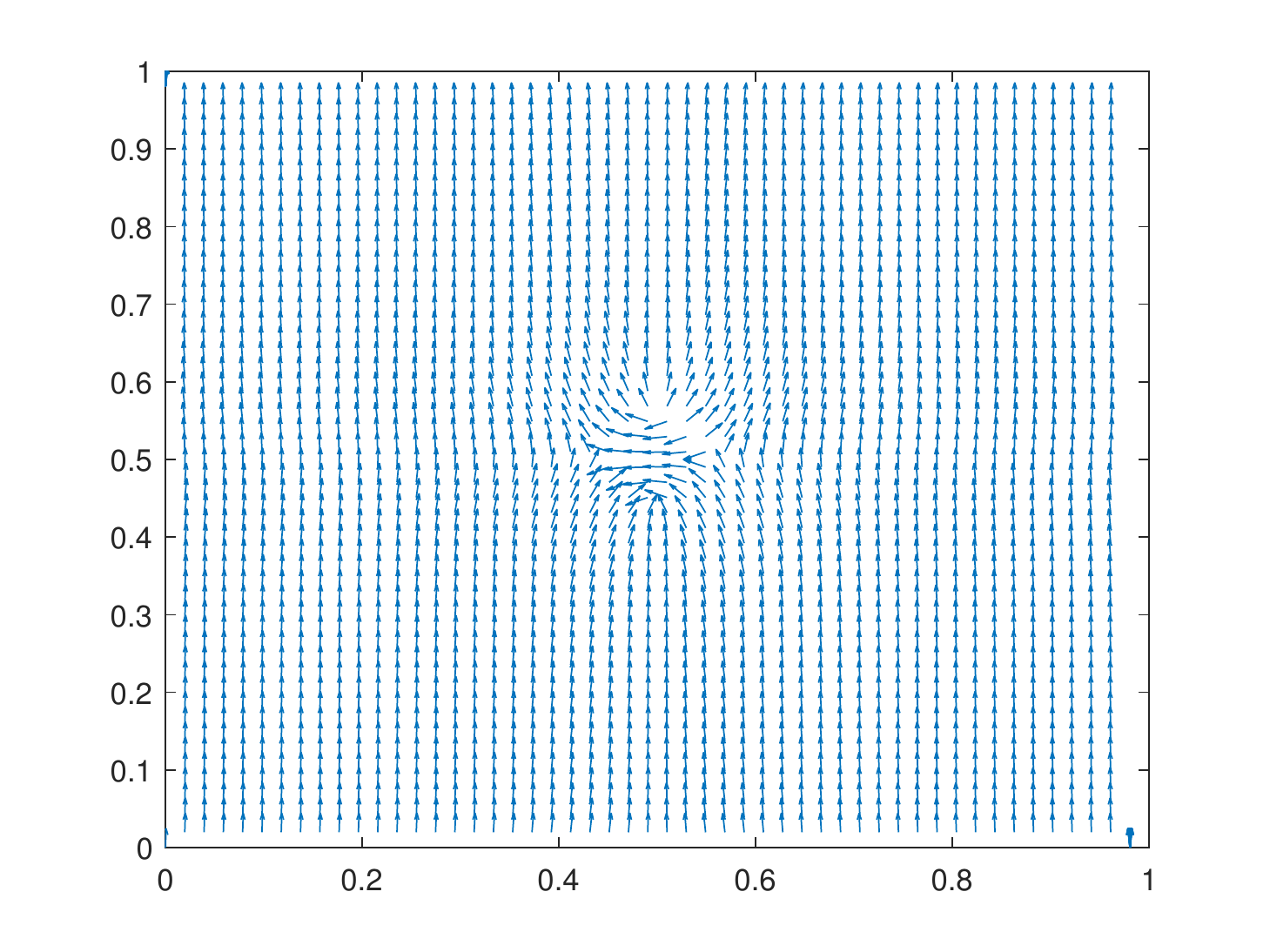}} &
		\subfloat[$t=0.8$]{\includegraphics[width = 2in]{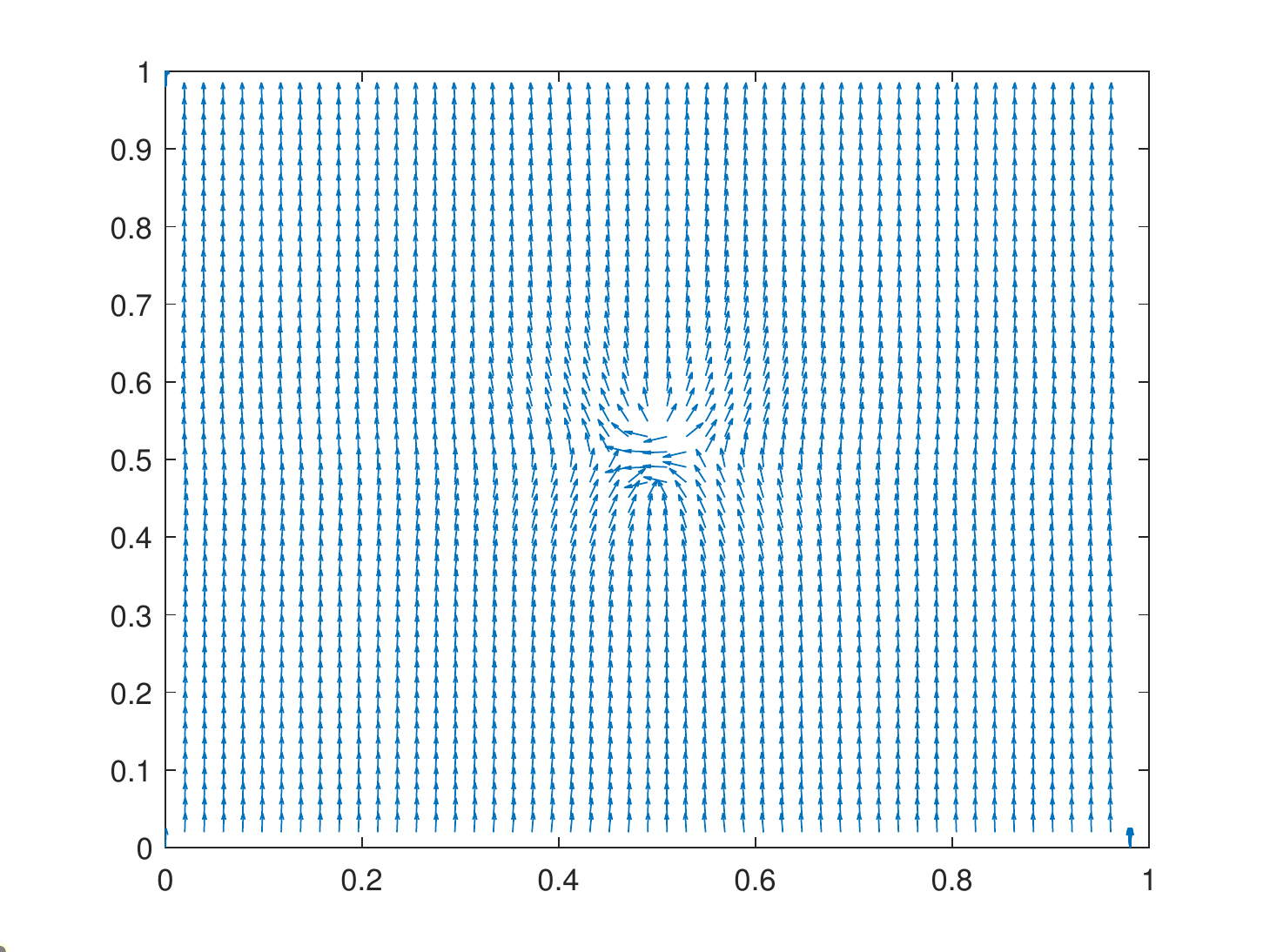}} &
		\subfloat[$t=1$]{\includegraphics[width = 2in]{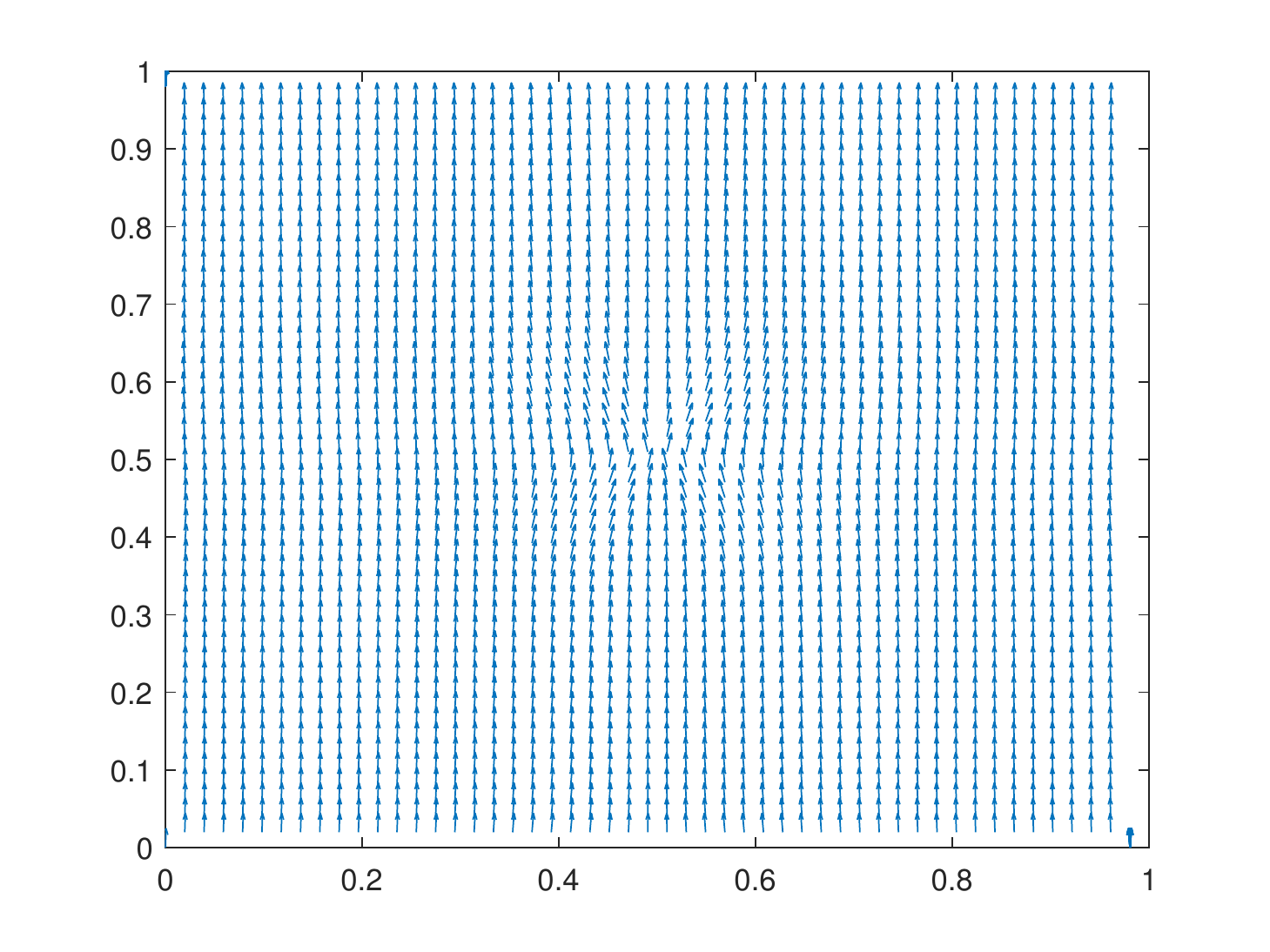}}
	\end{tabular}
	\caption{Simulation for initial data~\eqref{eq:q0}, \eqref{eq:init1}.}
	\label{fig:hole}
\end{figure}

\begin{figure}
    \centering
    \includegraphics[width=4in]{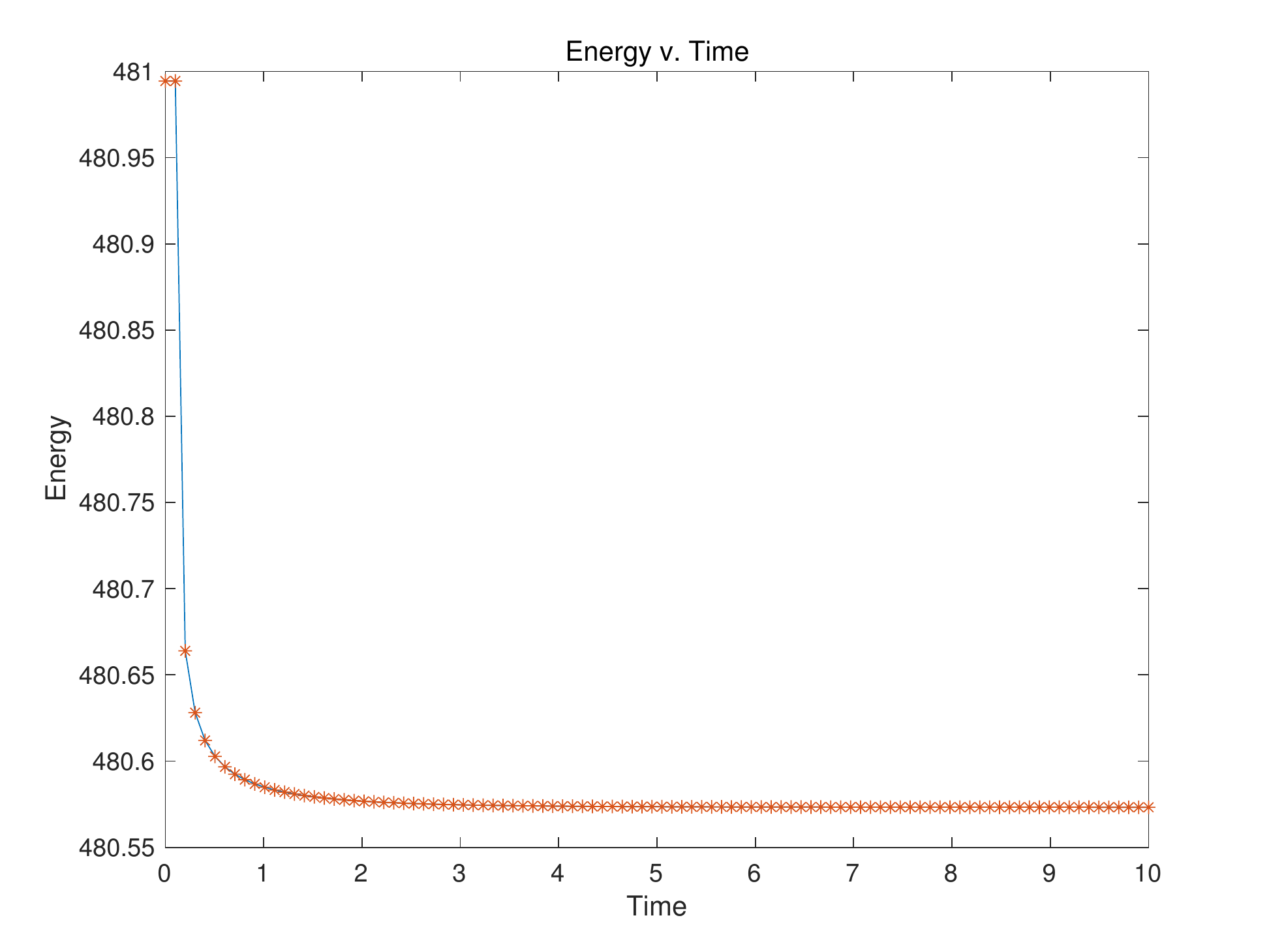}
    \caption{Energy when $T=10$ with $100$ time steps and 50 grid points in each spatial direction.}
    \label{disappearingholeenergey}
\end{figure}

\section{Acknowledgements}
We thank Max Hirsch for the careful reading of our manuscript and pointing out several mistakes and typos. F.W. and Y.Y. acknowledge partial funding by NSF awards DMS No.\ 1912854 and  OIA-DMR No.\ 2021019 and V.G. was supported in part by a NASA internship from the Pennsylvania Space Grant Consortium, no. NNX15AK06H.
\appendix
\section{Some Lemmas}
\begin{lemma} \label{a:sbp}
Let $A_{ijk}$ and $B_{ijk}$ be scalar quantities at grid point $(x_i, y_j, z_k)$ such that $A_{ijk} = 0$ at boundary values, i.e. boundary conditions~\eqref{eq:dirchlet},~\eqref{eq:ghostpoints}.
Then 
\begin{equation*}
\sum_{i,j,k=0}^{N+1} A_{ijk} D_\beta^+ B_{ijk} = -\sum_{i,j,k=0}^{N+1} B_{ijk} D_\beta^- A_{ijk},\quad  
\sum_{i,j,k=0}^{N+1} A_{ijk} D_\beta^- B_{ijk} = -\sum_{i,j,k=0}^{N+1} B_{ijk} D_\beta^+ A_{ijk},
\end{equation*}
and
\begin{equation*}
    \sum_{i,j,k=0}^{N+1} A_{ijk} D_\beta^c B_{ijk} = -\sum_{i,j,k=0}^{N+1} B_{ijk} D_\beta^c A_{ijk}.
\end{equation*}
for $\beta = 1, 2 \text{ or } 3$.
\end{lemma}

\begin{proof} We shall prove this for the case where $\beta = 1$, the other cases follow similarly. Note that 
	\begin{align*}
	\sum_{i,j,k=0}^{N+1} A_{ijk} D_1^+ B_{ijk}
	&= \frac{1}{h} \left( \sum_{i,j,k=1}^N A_{ijk} B_{(i+1)jk}- \sum_{i,j,k=1}^N A_{ijk} B_{ijk} \right) \\
	&= \frac{1}{h} \left( \sum_{j,k=0}^{N+1}\sum_{i=2}^{N+1} A_{(i-1)jk} B_{ijk} - \sum_{j,k=0}^{N+1}\sum_{i=1}^N A_{ijk} B_{ijk} \right) \\
	&= \frac{1}{h} \left( \sum_{i,j,k=0}^{N+1} A_{(i-1)jk}B_{ijk} - \sum_{i,j,k=0}^{N+1} A_{ijk}B_{ijk} \right) \\
	&= -\sum_{i,j,k=0}^{N+1} B_{ijk}D_1^- A_{ijk}
	\end{align*}
	where we used the boundary conditions ~\eqref{eq:dirchlet} and ~\eqref{eq:ghostpoints} for $A_{ijk}$. For the second identity, using the same trick, we obtain
	\begin{align*}
	\sum_{i,j,k=0}^{N+1} A_{ijk} D_1^- B_{ijk}
	&= \frac{1}{h} \left( \sum_{i,j,k=1}^N A_{ijk}B_{ijk} - \sum_{i,j,k=1}^N A_{ijk} B_{(i-1)jk} \right) \\
	&= \frac{1}{h} \left( \sum_{j,k=0}^{N+1}\sum_{i=1}^N A_{ijk}B_{ijk} - \sum_{j,k=0}^{N+1}\sum_{i=0}^{N-1} A_{(i+1)jk}B_{ijk} \right) \\
	&= \frac{1}{h} \left(\sum_{i,j,k=0}^{N+1} A_{ijk}B_{ijk} - \sum_{i,j,k=0}^{N+1} A_{(i+1)jk}B_{ijk} \right) \\
	&= -\sum_{i,j,k=0}^{N+1} B_{ijk}D_1^+ A_{ijk}.
	\end{align*}
	For the third identity,
		\begin{align*}
	\sum_{i,j,k=0}^{N+1} A_{ijk} D_1^c B_{ijk}
	&= \frac{1}{2h} \left( \sum_{j,k=0}^{N+1}\sum_{i=1}^N A_{ijk}B_{{(i+1)}jk} - \sum_{j,k=0}^{N+1}\sum_{i=1}^N A_{ijk} B_{(i-1)jk} \right) \\
	&= \frac{1}{2h} \left( \sum_{j,k=0}^{N+1}\sum_{i=2}^{N+1} A_{
	(i-1)jk}B_{ijk} - \sum_{j,k=0}^{N+1}\sum_{i=0}^{N-1} A_{(i+1)jk}B_{ijk} \right) \\
	&= \frac{1}{2h} \left(\sum_{i,j,k=0}^{N+1} A_{(i-1)jk}B_{ijk} - \sum_{i,j,k=0}^{N+1} A_{(i+1)jk}B_{ijk} \right) \\
	&= -\sum_{i,j,k=0}^{N+1} B_{ijk}D_1^c A_{ijk}.
	\end{align*}
	where we have used boundary values of $A_{ijk}$ and $B_{ijk}$. 
\end{proof}

We believe the following lemma is a standard result from real analysis but we did not find a suitable reference to refer to and therefore provide the proof here for completeness.
\begin{lemma}\label{lem:weakcont}
	Assume that $\{g_{h,\Delta t}\}_{h,\Delta t}$ is a sequence of piecewise constant functions converging weak*, as $h,\Delta t\to 0$, in $L^\infty([0,T];L^2(\dom))$ to some limit $g\in L^\infty([0,T];L^2(\dom))$ that is weakly continuous in time in $L^1(\dom)$, i.e., $\int g(s,x)\phi(x) dx \rightarrow \int g(t,x)\phi(x)dx$ when $s\to t$ for $\phi\in L^\infty(\dom)$. In addition, assume that
	\begin{equation*}
	\norm{D_t^+ g_{h,\Delta t}}_{L^2([0,T];L^1(\dom))}\leq C,
	\end{equation*}
	where $C$ is a constant independent of $h$ and $\Delta t$.
	Then, up to a subsequence,
	\begin{equation*}
	\int_{\dom} g_{h,\Delta t}\phi(x) dx \stackrel{h,\Delta t\to 0}{\longrightarrow} \int_{\dom} g(t,x)\phi(x) dx,
	\end{equation*}
	for all $t\in [0,T]$ and $\phi\in L^\infty(\dom)$.
\end{lemma}
\begin{proof}
	Let $\phi\in L^\infty(\dom)$. As $\{g_{h,\Delta t}\}$ is weak* convergent in $L^\infty([0,T];L^2(\dom))$ we can find a dense set $\mathcal{T}:=\{t_i\}_{i=1}^\infty\subset [0,T]$ such that for a (diagonal) subsequence $\{h_m,\Delta t_m\}_{m=1}^\infty$
	\begin{equation*}
	\int_{\dom} g_{h_m,\Delta t_m}(t_i,x)\phi(x) dx\stackrel{m\to\infty}{\longrightarrow} \int_{\dom} g(t_i,x)\phi(x) dx,\quad \text{	for all }\, t_i\in\mathcal{T}.
	\end{equation*}
	Fix $\epsilon>0$ arbitrary and $t\in [0,T]$. Then since $g$ is weakly continuous, we can find an interval $I\subset [0,T]$ such that $t\in I$ and for all $s\in I$,
	\begin{equation*}
	\left|\int_{\dom} g(t,x)\phi(x) dx -\int_{\dom}g(s,x)\phi(x) dx\right|<\frac{\epsilon}{3}.
	\end{equation*}
	Next, we pick $M_1\in\N$ large enough, such that for all $m\geq M_1$, and all $t_j\in \mathcal{T}\cap I$,
	\begin{equation*}
	\left|\int_{\dom} g_{h_m,\Delta t_m}(t_j,x)\phi (x)dx -\int_{\dom} g(t_j,x)\phi (x)dx\right|<\frac{\epsilon}{3}.
	\end{equation*}
	We observe that we can write for $s\geq t\in [0,T]$,
	\begin{equation*}
	\int_{\dom}\left(g_{h,\Delta t}(s,x)-g_{h,\Delta t}(t,x)\right)\phi(x) dx 
	 =\Delta t \int_{\dom}\left(\sum_{\ell=\floor*{\frac{t}{\Delta t}}}^{\floor*{\frac{s}{\Delta t}}-1} D_t^+ g_{h}^\ell(x)\right)\phi(x) dx.
	\end{equation*}
	Thus,
	\begin{align*}
\left|	\int_{\dom}\left(g_{h,\Delta t}(s,x)-g_{h,\Delta t}(t,x)\right)\phi(x) dx	\right|	& \leq\Delta t \int_{\dom}\left(\sum_{\ell=\floor*{\frac{t}{\Delta t}}}^{\floor*{\frac{s}{\Delta t}}-1} \left|D_t^+ g_{h}^\ell(x)\right|\right)|\phi(x)| dx\\
& \leq \Delta t \sum_{\ell=\floor*{\frac{t}{\Delta t}}}^{\ceil*{\frac{s}{\Delta t}}-1} \norm{D_t^+ g_{h}^\ell}_{L^1(\dom)}\norm{\phi}_{L^\infty} \\
& \leq \Delta t \left(\sum_{\ell=\floor*{\frac{t}{\Delta t}}}^{\ceil*{\frac{s}{\Delta t}}-1} \norm{D_t^+ g_{h}^\ell}_{L^1(\dom)}^2\right)^{1/2}\left(\ceil*{\frac{s}{\Delta t}}-\floor*{\frac{t}{\Delta t}}\right)^{1/2}\norm{\phi}_{L^\infty} \\
&\leq  \left(\Delta t\sum_{\ell=\floor*{\frac{t}{\Delta t}}}^{\ceil*{\frac{s}{\Delta t}}-1} \norm{D_t^+ g_{h}^\ell}_{L^1(\dom)}^2\right)^{1/2}\left(s-t+\Delta t\right)^{1/2}\norm{\phi}_{L^\infty} \\
&\leq  \norm{D_t^+ g_{h,\Delta t}}_{L^2([0,T];L^1(\dom))}\left(s-t+\Delta t\right)^{1/2}\norm{\phi}_{L^\infty}\\
&\leq C \left(s-t+\Delta t\right)^{1/2}.
	\end{align*}
	So we pick $M_2\geq M_1$ large enough and $J\subset I$ such that for $m\geq M_2$ and $t_j\in J$,
	\begin{equation*}
	\left|	\int_{\dom}\left(g_{h_m,\Delta t_m}(t_j,x)-g_{h,\Delta t}(t,x)\right)\phi(x) dx	\right|\leq C \left(t_j-t+2\Delta t_m\right)^{1/2}<\frac{\epsilon}{3}.
	\end{equation*}
	Then we have for $m\geq M_2$ (and $t_j\in J$),
	\begin{equation*}
	\begin{split}
	\left|\int_{\dom}(g(t,x)-g_{h_m,\Delta t_m}(t,x))\phi(x)dx\right|&\leq \left|\int_{\dom}(g(t,x)-g(t_j,x))\phi(x)dx\right|\\
	&\quad  + \left|\int_{\dom}(g(t_j,x)-g_{h_m,\Delta t_m}(t_j,x))\phi(x)dx\right| \\
	&\quad + \left|\int_{\dom}(g_{h_m,\Delta t_m}(t_j,x)-g_{h_m,\Delta t_m}(t,x))\phi(x)dx\right|\\
	&\leq \epsilon,
	\end{split} 
	\end{equation*}
	which proves the result.
\end{proof}

\bibliographystyle{abbrv}
\bibliography{qtensor_refs}

\begin{thebibliography}{10}

\bibitem{Ball2017}
J.~M. Ball.
\newblock Mathematics and liquid crystals.
\newblock {\em Molecular Crystals and Liquid Crystals}, 647(1):1--27, 2017.

\bibitem{Beris1994}
A.~Beris, B.~Edwards, B.~Edwards, and C.~Edwards.
\newblock {\em Thermodynamics of Flowing Systems: With Internal
  Microstructure}.
\newblock Oxford engineering science series. Oxford University Press, 1994.

\bibitem{Cai2017}
Y.~Cai, J.~Shen, and X.~Xu.
\newblock A stable scheme and its convergence analysis for a 2{D} dynamic
  {$Q$}-tensor model of nematic liquid crystals.
\newblock {\em Math. Models Methods Appl. Sci.}, 27(8):1459--1488, 2017.

\bibitem{Contreras2019}
A.~Contreras, X.~Xu, and W.~Zhang.
\newblock An elementary proof of eigenvalue preservation for the co-rotational
  beris-edwards system.
\newblock {\em Journal of Nonlinear Science}, 29(2):789--801, 2019.

\bibitem{deGennes1995}
P.~de~Gennes and J.~Prost.
\newblock {\em The Physics of Liquid Crystals}.
\newblock International Series of Monogr. Clarendon Press, 1995.

\bibitem{Goodby1989}
J.~W. Goodby, E.~Chin, and J.~S. Patel.
\newblock Eutectic mixtures of ferroelectric liquid crystals.
\newblock {\em The Journal of Physical Chemistry}, 93(24):8067--8072, 1989.

\bibitem{Iyer2015}
G.~Iyer, X.~Xu, and A.~D. Zarnescu.
\newblock {Dynamic cubic instability in a 2D Q-tensor model for liquid
  crystals}.
\newblock {\em Mathematical Models and Methods in Applied Sciences},
  25(08):1477--1517, 2015.

\bibitem{Mori1999}
H.~Mori, E.~C. Gartland, J.~R. Kelly, and P.~J. Bos.
\newblock {Multidimensional Director Modeling Using the Q Tensor Representation
  in a Liquid Crystal Cell and Its Application to the Pi-Cell with Patterned
  Electrodes}.
\newblock {\em Japanese Journal of Applied Physics}, 38(Part 1, No.
  1A):135--146, jan 1999.

\bibitem{Mottram2014}
N.~J. {Mottram} and C.~J.~P. {Newton}.
\newblock {Introduction to Q-tensor theory}.
\newblock {\em ArXiv e-prints}, Sept. 2014.

\bibitem{shen2019}
J.~Shen, J.~Xu, and J.~Yang.
\newblock A new class of efficient and robust energy stable schemes for
  gradient flows.
\newblock {\em SIAM Review}, 61(3):474--506, 2019.

\bibitem{sv2012}
A.~M. Sonnet and E.~Virga.
\newblock {\em {Dissipative Ordered Fluids, Theories for Liquid Crystals}}.
\newblock Springer US, 2012.

\bibitem{Temam_NS}
R.~Temam.
\newblock {\em Navier Stokes Equations: Theory and Numerical Analysis}, volume
  318.
\newblock AMS Chelsea Publishing, 1985.

\bibitem{Tovkach2017}
O.~M. Tovkach, C.~Conklin, M.~C. Calderer, D.~Golovaty, O.~D. Lavrentovich,
  J.~Vi\~nals, and N.~J. Walkington.
\newblock Q-tensor model for electrokinetics in nematic liquid crystals.
\newblock {\em Phys. Rev. Fluids}, 2:053302, May 2017.

\bibitem{Trivisa2017}
K.~Trivisa and F.~Weber.
\newblock A convergent explicit finite difference scheme for a mechanical model
  for tumor growth.
\newblock {\em ESAIM Math. Model. Numer. Anal.}, 51(1):35--62, 2017.

\bibitem{WangWang2017}
M.~{Wang}, W.~{Wang}, and Z.~{Zhang}.
\newblock {From the Q-Tensor Flow for the Liquid Crystal to the Harmonic Map
  Flow}.
\newblock {\em Archive for Rational Mechanics and Analysis}, 225(2):663--683,
  Aug. 2017.

\bibitem{Wissbrun1984}
K.~F. Wissbrun.
\newblock Orientation development in liquid crystal polymers.
\newblock In J.~L. Ericksen, editor, {\em Orienting Polymers: Proceedings of a
  Workshop held at the IMA, University of Minnesota, Minneapolis March 21--26,
  1983}, pages 1--26, Berlin, Heidelberg, 1984. Springer Berlin Heidelberg.

\bibitem{ZhaoYang2017}
J.~Zhao, X.~Yang, Y.~Gong, and Q.~Wang.
\newblock A novel linear second order unconditionally energy stable scheme for
  a hydrodynamic {Q}-tensor model of liquid crystals.
\newblock {\em Comput. Methods Appl. Mech. Engrg.}, 318:803--825, 2017.

\end{thebibliography}

\end{document}